\documentclass[a4paper]{amsart}
\usepackage{amssymb, amsmath, amsthm} 

\usepackage[utf8]{inputenc}
\usepackage[T1]{fontenc} 

\usepackage{microtype}
\usepackage{bbm}

\usepackage{mathtools}
   \mathtoolsset{centercolon} 
   
 \usepackage[shortlabels]{enumitem}

\usepackage[hidelinks]{hyperref} 
   \hypersetup{final} 

\newtheorem{theorem}{Theorem}
\newtheorem{lemma}[theorem]{Lemma}
 
\newtheorem{proposition}[theorem]{Proposition} 

\newtheorem*{theorem*}{Theorem}

 \theoremstyle{definition}
\newtheorem{definition}[theorem]{Definition}
 \newtheorem{example}[theorem]{Example} 

\theoremstyle{remark}
\newtheorem{remark}[theorem]{Remark}

\newtheorem*{remark*}{Remark}

\numberwithin{equation}{section}

    \newcommand*{\RR}{\mathbb{R}}
        
    \newcommand*{\NN}{\mathbb{N}}
     \newcommand*{\ZZ}{\mathbb{Z}}
     \newcommand*{\CC}{\mathbb{C}}
    \newcommand*{\EE}{\mathbb{E}}
    \newcommand*{\PP}{\mathbb{P}}

    \newcommand*{\eps}{\varepsilon}

    \newcommand*{\ind}{\mathbf{1}}

\makeatletter
\@namedef{subjclassname@2020}{%
  \textup{2020} Mathematics Subject Classification}
\makeatother

\begin{document}

\title[MDPs \& Mod-Gaussian convergence for lacunary sums]{Moderate deviation principles and Mod-Gaussian convergence for lacunary trigonometric sums}

\author[J. Prochno]{Joscha Prochno}
\address{Faculty of Computer Science and Mathematics, University of Passau, Dr.-Hans-Kapfinger-Stra{\ss}e 30, 94032 Passau, Germany}
\email{joscha.prochno@uni-passau.de}
\thanks{}

\author[M. Strzelecka]{Marta Strzelecka}
\address{Institute of Mathematics, University of Warsaw,  Banacha 2, 02--097 Warsaw, Poland.}
\email{martast@mimuw.edu.pl}
\thanks{}

\begin{abstract}
Classical works of Kac, Salem and Zygmund, and Erd\H os and G\'al have shown that lacunary trigonometric sums despite their dependency structure behave in various ways like sums of independent and identically distributed random variables. For instance, they satisfy a central limit theorem and a law of the iterated logarithm. Those results have only recently been complemented by large deviation principles by Aistleitner, Gantert, Kabluchko, Prochno, and Ramanan, showing that interesting phenomena occur on the large deviation scale that are not visible in the classical works. This raises the question on what scale such phenomena kick in. In this paper, we provide a first step towards a resolution of this question by studying moderate deviation principles for lacunary trigonometric sums. We show that no arithmetic affects are visible between the CLT scaling $\sqrt{n}$ and a scaling $n/\log(n)$ that is only a logarithmic gap away from the large deviations scale. To obtain our results, inspired by the notion of a dependency graph, we introduce correlation graphs and use the method of cumulants. In this work we also obtain results on the mod-Gaussian convergence using different tools.\end{abstract}



\maketitle

\section{Introduction \& main results}

It is today well-known in probabilistic number theory that, from an asymptotic point of view, lacunary trigonometric sums $\sum_{k=1}^n\cos(2\pi a_k x)$, $x\in [0,1]$, behave in various ways like sums of independent and identically distributed (i.i.d.) random variables if the sequence $(a_k)_{k\in\NN}$ of natural numbers is lacunary in the sense that it satisfies the so-called Hadamard gap condition
  \begin{equation}\label{eq:lacunary-condition}
    \frac{a_{k+1}}{a_k} \geq q > 1.
  \end{equation}
It should be noted that $\cos(2\pi a_k\cdot)$, $k\in\NN$, considered as random variables on the probability space $[0,1]$ endowed with the Borel $\sigma$-field and Lebesgue measure $\lambda$ are identically distributed, but \textit{merely} uncorrelated (because the $a_k$'s are distinct) and \textit{not} independent.
In some cases we also refer to the so-called large gap condition
\begin{equation}
	\frac{a_{n+1}}{a_n} \mathop{\longrightarrow}^{n\to \infty} \infty,
\label{eq:large-gap-condition}
\end{equation}
which means that we have super-exponential gaps. There are a number of classical works studying the probabilistic nature of lacunary (trigonometric) sums. For instance, Salem and Zygmund \cite{SalemZyg1950} and  Erd\H{o}s and G\'al \cite{EG1955} obtained that under the Hadamard gap condition \eqref{eq:lacunary-condition},
  \begin{equation*} 
    \limsup_{N\to\infty} \frac{\bigl| \sum_{k=1}^N \cos(2\pi n_kx) \bigr|}{\sqrt{2 N\log\log N}} = \frac{1}{\sqrt{2}} \quad \text{a.e.},
  \end{equation*}
thereby showing that lacunary trigonometric sums satisfy a Hartman--Wintner law of the iterated logarithm (LIL); for further results regarding a LIL, we refer the reader to  \cite{A2010} and \cite{AFP23}). Regarding Gaussian fluctuations, Salem and Zygmund obtained in \cite{SZ-1} that under the Hadamard gap condition \eqref{eq:lacunary-condition}, for every $t\in\mathbb R$, lacunary trigonometric sums satisfy the central limit theorem (CLT for short)
  \[
    \lim_{N\to\infty} \lambda \biggl( \biggr\{ x\in[0,1]\,:\, \frac{\sum_{k=1}^N\cos(2\pi n_kx)}{\sqrt{N/2}} \leq t \biggr\} \biggr) = \frac{1}{\sqrt{2\pi}}\int_{-\infty}^t e^{-y^2/2}\,\text{d}y,
  \]
which resembles the CLT for sums of independent random variables; a~CLT in the case of the large gap condition \eqref{eq:large-gap-condition} had previously been shown by Kac in \cite{K1939}. It is of course a natural question to ask whether the LIL and the CLT still hold under the gap condition when the function $\omega\mapsto \cos(2\pi\omega)$ is replaced by an arbitrary Lipschitz continuous, centered and $1$-periodic function $f:\RR\to\RR$. However, as a famous example of Erd\H{o}s and Fortet (see, e.g., \cite{Kac49, ABT23}) shows, this is not true in general. We refer to Kac's work \cite{KacCLT} for a central limit theorem in the setting of the lacunary sequence $a_n=2^n$  and more general centered  $1$-periodic functions.

\subsection{Recent contributions \& motivation}

Let us now go beyond the classical works. We continue with our setting and present some more recent and complementing results, some of which go beyond the trigonometric framework. 
They  serve as a motivation for our paper.  

Let $D\in\NN$ and $f\colon \RR \to \RR$ be a trigonometric polynomial of the form
\begin{equation}\label{eq:trigonometric polynomial}
	f(x) = \sum_{d=-D}^D c_d e^{2\pi idx}.
\end{equation}
We assume that the coefficients $c_d\in\RR$, $d\in\{-D,\dots,D\}$ satisfy the symmetry condition $c_d=c_{-d}$;  this is equivalent to $f$ being a real trigonometric polynomial. Moreover, we assume that $\sum_{d=-D}^D c_d^2 >0$, because otherwise $f\equiv 0$, and, without loss of generality, we let $c_0=\int_0^1f(x)\,\text{d}x=0$, which means that $f$ is centered; otherwise we may simply shift all the considered random variables by the deterministic constant $c_0$. Note also that $f$ is clearly a $1$-periodic function. Let $U\sim \operatorname{Unif}(0,1)$ and $X_j=f(a_jU)$, $j\in\NN$.
Since $f$ is $1$-periodic and  $a_j$'s are integers, every $X_j$ has the same distribution as $f(U)$.
Moreover, every $X_j$ is centered, because we assume that $c_0=0$.
 For any $n\in\NN$, we write 
\[
	S_n^f\coloneqq \sum_{j=1}^n X_j = \sum_{j=1}^n f(a_jU).
\]
This is a sum of identically distributed random variables, which are dependent. However, as we shall see once more, it sometimes behaves in a sense like a sum of i.i.d.~random variables $Y_j \coloneqq  f(U_j)$, i.e., like
	\[	
		T_n^f \coloneqq \sum_{j=1}^n Y_j =  \sum_{j=1}^n f(U_j) \stackrel{d}{=} \sum_{j=1}^n f(a_jU_j),
	\]
	where $(U_j)_{j\in\NN}$ is a sequence of  i.i.d.~random variables distributed uniformly on $[0,1]$. If $f:\RR\to\RR$ is given by $f(x)=\cos(2\pi x)$, then the variance of $S_n$ is $\frac{n}{2}$ and by the previously mentioned CLT of Salem and Zygmund \cite{SZ-1}, we know that 
  \[
    \frac{S_n}{\sqrt{n/2}} = \frac{S_n}{\sqrt{\EE S_n^2}} \stackrel{d}{\longrightarrow} \mathcal N(0,1),
  \]
i.e., we have convergence in distribution to the standard Gaussian law. However, for general trigonometric polynomials $f$ the CLT may not hold, which shows that the limiting behavior of $S_n$ depends on the analytic properties of the function $f$. For example, if $f(x)=\cos(2\pi x)-\cos(4\pi x)$ and $a_n=2^n$, then $S_n^f$ is a telescoping sum. Hence, the CLT cannot hold because of the trivial reason that the variances of $S_n^f$, $n\in\NN$, are equal to $1$, so they degenerate, i.e., they are of smaller order than in the i.i.d.~case, where they are of order $n$. As  it turns out, this is the only reason why the CLT with the Gaussian limit cannot hold in the case of the geometric progression $a_n=2^n$, $n\in\NN$. This is a consequence of the following much more general characterization via solutions to two-variable linear Diophantine equations obtained by Aistleitner and Berkes {\cite[Theorems~1.1 and 1.3]{AB2010}}. 
	
	\begin{theorem*}\label{thm:AB}
		Let $a=(a_n)_{n\in\NN}$ be a lacunary sequence. Then the following are equivalent:
		\begin{enumerate}
			\item $\bigl(\frac{S_n^f}{\sqrt{\EE S_n^2}}\bigr)_{n\in\NN}$ converges in distribution to the standard Gaussian law for every centered trigonometric polynomial $f$ satisfying $\EE (S_n^f)^2\ge c(f,a) n$ for all $n\in\NN$, where $c(f,a)\in(0,\infty)$ is a constant depending only on $f$ and the lacunary sequence $a$.
			\item For every $b,c\in \ZZ$,
			\begin{equation}	\label{eq:cond-AB}
				\sup_{d\in\ZZ\setminus \{0\}} \#\{ (k,\ell)\in \{1,\ldots,n\}^2 \colon ba_k+ca_\ell=d \} = o(n). 
			\end{equation}
		\end{enumerate}
	\end{theorem*}
	
	Let us note that there are examples of lacunary sequences violating condition \eqref{eq:cond-AB}, but for which $\bigl(\frac{S_n}{\sqrt{\EE S_n^2}}\bigr)_{n\in\NN}$ still converges, but then the limiting distribution is of course different from the standard Gaussian law (cf. \cite[Section~4]{ABT23}). 
	Other directions of generalizing CLT results for trigonometric lacunary sums were investigated, for instance, in \cite{Ber79, BG07}.
	
	The theorem of Aistleitner and Berkes already reveals the impact of the algebraic 
	and arithmetic properties of the lacunary sequence $(a_n)_{n\in\NN}$ on the limiting behavior of $S_n$ in the CLT scaling. 
	As a matter of fact, it was quite surprising to see that the limiting behavior of $S_n$ in the large deviation principle (LDP for short) regime already in the cosine setting is more sensitive even for minor changes in the sequence $(a_n)_{n\in\NN}$, as is revealed by theorem below, which was only recently obtained by Aistleitner, Gantert, Kabluchko, Prochno, and Ramanan \cite{agkpr2020}, and by Fr{\"u}hwirth, Juhos, and Prochno \cite{FJP};\footnote{This theorem was first proven in \cite{agkpr2020} in the classical case when $f$ is the cosine, and then generalized to the case of general $1$-periodic Lipschitz functions in \cite{FJP}.} we postpone a formal definition of an LDP until later when we present our main results.
	
	\begin{theorem*}
	\label{thm:AGKPR} 
		\begin{enumerate} 
			\item Assume that $(a_n)_{n\in\NN}$ satisfies the large gap condition \eqref{eq:large-gap-condition} and $f$ is a  $1$-periodic Lipschitz function. Then $(\frac{S_n^f}n)_{n\in\NN}$ satisfies an LDP with speed $n$ and a good rate function $I_f(t)=\Lambda_{X_1}^*(t)$.
			\item Assume that $q\in \NN\setminus\{1\}$, $a_n=q^n$, and  $f$ is a  $1$-periodic Lipschitz function. Then the  sequence of functions $(\frac 1n \log \EE e^{zS_n^f})_{n\in\NN}$ converges locally uniformly on $\CC$ to a function $J_{q,f}$ depending only on $q$ and $f$.
			Moreover, $(\frac{S_n^f}n)_{n\in\NN}$ satisfies an LDP with speed $n$ and a good rate function $I_{q,f}(t)=J_{q,f}^*(t)$.
			\item If $f(x)=\cos(2\pi x)$, then $I_{p,f}\neq I_{q,f}\neq \Lambda^*_{X_1}$ for every pair of distinct $p,q\in \NN\setminus \{1\}$.
			\item If $f(x)=\cos(2\pi x)$, then there exists a random sequence of integers $(a_n)_{n\in\NN}$ satisfying the Hadamard gap condition \eqref{eq:lacunary-condition} a.s. as well as the inequality $|a_n-2^n|\leq 2^{n^{2/3}}$ a.s., and such that $(\frac{S_n}n)_{n\in\NN}$ satisfies an LDP with speed $n$ and a good rate function $I(t)=\Lambda_{X_1}^*(t)$.
		\end{enumerate}
	\end{theorem*}
	
	This theorem shows that the limiting behavior of $(\frac{S_n^f}n)_{n\in\NN}$ may change
	drastically for small random perturbations of a lacunary sequence, and is then different than the behavior of the rescaled sum of i.i.d.~random variables $\frac{T_n^f}n$, $n\in\NN$, even for very nice geometric sequences $(a_n)_{n\in\NN}$. 
	
Let us pose two natural questions: 
	\begin{enumerate}[(Q1)]
		\item \label{item:Q1}
		Assume that $f(x)=\cos(2\pi x)$. Does there exist a threshold for scaling $y_n$, below which the limiting behavior of $(\frac{S_n}{y_n})_{n\in\NN}$ remains the same as in the i.i.d.~case for every geometric lacunary sequence and for every sequence satisfying the large gap condition, and above of which this limiting behavior varies even for different geometric lacunary sequences?
		\item  \label{item:Q2}
		Assume that $f$ is a $1$-periodic function, and $a=(a_n)_{n\in\NN}$ is such a lacunary sequence satisfying \eqref{eq:cond-AB} that the variances of $S_n^f$ do not degenerate, i.e., they satisfy 
		\begin{equation} \label{condition-var}
		\EE (S_n^f)^2 \ge c(f,a) n  
	\end{equation}
	for every $n\in\NN$ and some constant $c(f,a)\in(0,\infty)$ only depending on $f$ and $a$, but not on $n$.
		What is the limiting behavior of $(\frac{S_n^f}{y_n})_{n\in\NN}$ when $y_n>0$ converges to infinity fast enough to ensure that $\lim_{n\to \infty}    \frac{y_n}{  \sqrt n  }  =\infty $ (scale of larger order than the one in the CLT), but slow enough to ensure that $\lim_{n\to \infty}   \frac{y_n}n=  0 $ (scale of smaller order than the one in the LDP)? 
	\end{enumerate}
	
\subsection{Main results}

	 The first main contribution of this article are the following two moderate deviation principle (MDP for short) results obtained in the case when all ratios $\frac{a_{n+1}}{a_n}$, $n\in\NN$, are natural numbers or when $(a_n)_{n\in\NN}$ satisfies the large gap condition \eqref{eq:large-gap-condition}; note that the case of integer ratios includes the  class of geometric progressions $a_n= q^n$ with  ratio $q\in\{2,3,\dots\}$. 
	 These results provide partial answer to questions (Q1) and (Q2) above, a fact on which we shall elaborate further after stating the results. 
	 
	 Let us briefly recall that we say that the sequence of real random variables $(Z_n)_{n\in\NN}$ satisfies an LDP with speed $x_n$ and a good rate function $I\colon \RR\to [0,\infty]$ if and only if $\lim_{n\to \infty}x_n= \infty$, $x_n>0$, all the sublevel sets of $I$  are compact, and for every Borel measurable set $B$ in $\RR$, 
	\[
		-\inf_{t\in \operatorname{int}B} I(t) \le \liminf_{n\to \infty} \frac1{x_n} \log\PP(Z_n\in B)  \le \limsup_{n\to \infty} \frac1{x_n} \log\PP(Z_n\in B) \le -\inf_{t\in \overline{B}} I(t),
	\]
where $\operatorname{int}(B)$ denotes the interior of the set $B$ and $\overline{B}$ the closure of $B$.
	If $T_n=Y_1+\ldots+Y_n$ for i.i.d.~random variables $Y_1, Y_2, \ldots$, then by the classical Cram{\'e}r's theorem \cite{C1938}, $(\frac{T_n}n)_{n\in\NN}$ satisfies an LDP with speed $n$ and the good rate function $I$ being the Legendre-Fenchel transform of the cumulant generating function of $Y_1$, i.e., $I(t)=\Lambda_{Y_1}^*(t)=\sup_{s\in \RR} \{st - \Lambda_{Y_1}(s)\}$, where $\Lambda_{Y_1}(s)=\log \EE e^{sY_1}$.
	
	In the case when $Z_n = \frac{S_n}{y_n}$, where $S_n$ is a sum of $n$ random variables and $y_n$ converges to $\infty$ slower than $n$, but faster than $\sqrt n$, we rather say that $(\frac{S_n}{y_n})_{n\in\NN}$ satisfies a moderate deviation principle (MDP for short) --- in this case we reserve the term LDP for the limiting behavior of $ \frac{S_n}{n}$. In other words, an MDP corresponds to the scaling between the one of a CLT (where the scaling is $\sqrt n$) and the one of an LDP (where the scaling is $n$). 
	
	Our first main result is the following MDP. 
	 
	\begin{theorem}	\label{thm:integer-ratios}
 	Assume that $(z_n)_{n\in\NN}$ is a sequence of positive numbers converging to $0$ slow enough to ensure that 
	\[
		z_n n  \  \mathop{\longrightarrow}^{n\to \infty} \infty,
	\]
	but fast enough to ensure that
	\[
		z_n (\log n)^2  \  \mathop{\longrightarrow}^{n\to \infty} 0.
	\]
	Let $x_n=z_n\EE (S_n^f)^2$ and $y_n=\sqrt{z_n}\,\EE (S_n^f)^2$. 
	Assume that $\frac{a_{n+1}}{a_n}\in \{2, 3,\ldots\}$ for every $n\in\NN$,
	and  $f\not\equiv 0$ is an even trigonometric  polynomial with $\int_0^1 f =0$. 
	If condition \eqref{condition-var} is satisfied,
	 then $(\frac{S_n^f}{y_n})_{n\in\NN}$ satisfies an MDP with speed $x_n$ and good rate function $I(t)=\frac{t^2}{2}$.
	\end{theorem}

\begin{remark*}
Because the case of trigonometric lacunary sums is of special importance, we elaborate on the case $f(x)=\cos(2\pi x)$. In this case $\EE(S_n^f)^2=\EE S_n^2=\frac{n}{2}$ (see Example~\ref{ex:cosine} for a calculation), i.e., condition \eqref{condition-var} is satisfied,
so Theorem~\ref{thm:integer-ratios} applies and thus $(\frac{S_n}{y_n})_{n\in\NN}$ satisfies an MDP for every scaling sequence $y_n$ such that
  \[
    y_n=\omega\big(\sqrt{n}\big) \qquad\text{and}\qquad y_n=o\Big(\frac{n}{\log(n)}\Big).
  \]
  (For the speed we obtain $x_n=\omega(1)$ and $x_n=o\bigl(\frac{n}{\log^2(n)} \bigr)$, respectively.)
Hence, we are beyond a CLT scaling and may be logarithmically close to the one of an LDP. The result therefore shows that in this range no arithmetic affects are visible  if the ratios $\frac{a_{n+1}}{a_n}$ are integer (in contrast to the affects visible even for geometric progressions in the large deviations scale, as the aforementioned theorem due to \cite{agkpr2020, FJP} shows).
\end{remark*}

	Our second main result provides an MDP in the case of lacunary sequences with large gaps.
	
	\begin{theorem}	\label{thm:MDP-large-gap}
		Assume that $(z_n)_{n\in\NN}$ is a sequence of positive numbers converging to $0$ slow enough to ensure that 
		
	\[ 
		z_n n  \  \mathop{\longrightarrow}^{n\to \infty} \infty.
	\]
		Let $(a_n)_{n\in\NN}$ satisfy the large gap condition \eqref{eq:large-gap-condition}, $x_n=nz_n\EE X_1^2$, and $y_n=n\sqrt{z_n}\EE X_1^2$. Then $ \EE (S_n^f)^2=n\EE X_1^2 + O(1)$ and $(\frac{S_n^f}{y_n})_{n\in\NN}$ satisfies an MDP with speed $x_n$ and good rate function $I(t)=\frac{t^2}{2}$.
	\end{theorem}

	\begin{remark*}
	  Our two results show that the threshold from question~\ref{item:Q1} may exist only logarithmically close to the LDP scaling -- see the remark after Theorem~\ref{thm:integer-ratios}.
	  
	   Moreover, Remark~\ref{rmk:var-upper-bound} below implies that under  assumption \eqref{condition-var}, $\EE (S_n^f)^2$ is in fact of order $n$, so $z_n\equiv1$ would corresponds to the LDP scaling $y_n=\frac 1n$, while $z_n=\frac 1n$ would correspond to the CLT scaling $y_n=\frac1{\sqrt n}$, where the ``speed'' is constantly equal to $1$.
	  Therefore, Theorems~\ref{thm:integer-ratios} and \ref{thm:MDP-large-gap} answer question~\ref{item:Q2} in  the case when all the ratios $\frac{a_{n+1}}{a_n}$ are integers (then we need to assume additionally that the scaling sequence $y_n$ is  $o(n/\log n)$)  or when the large gap condition holds (with no additional restriction on the scaling sequence $y_n$).
	\end{remark*}
	
	\begin{remark*}
	 Although we use the method of cumulants in the proof of Theorem~\ref{thm:integer-ratios}, we are not able to obtain the bound for cumulants good enough to get the mod-Gaussian convergence for free, see also Remark~\ref{rmk:no-mod-G-for-free} below. However, we   are able to provide the mod-Gaussian convergence\footnote{The authors of \cite{DH18} --- a paper which seems to have never been published --- provide some results for lacunary trigonometric sums, but in their theorems a mod-Gaussian convergence refers to a weaker notion they introduce in \cite[Definition~1.8]{DH18}. Our result involves a standard, stronger notion of mod-Gaussian convergence.} under a \emph{very} large gap condition, which is unfortunately stronger than \eqref{eq:large-gap-condition}. We postpone the precise statement to Section~\ref{sect:mod-G}.
	  \end{remark*}
	  
	\vskip 1mm
	
	\textbf{Organization of the paper.} The rest of this paper is organized as follows. In Section~\ref{sect:MDP} we provide some preliminary results and then prove Theorem~\ref{thm:integer-ratios} by the method of cumulants, introducing the notion of correlation graphs. 
	At the end of Section~\ref{sect:MDP}  we give a proof of Theorem~\ref{thm:MDP-large-gap}. 
	In Section~\ref{sect:mod-G} we present our results concerning mod-Gaussian convergence and give their proofs.

\section{Moderate deviations via cumulants}\label{sect:MDP}

Let us begin with a quick warm-up, providing preliminary results which we shall use later in the proofs of the MDP and the mod-Gaussian convergence. Note that an iterated application of the Hadamard gap condition \eqref{eq:lacunary-condition} implies that for every $k\in\NN$ and any two indices $i,j \in \{1,\ldots, n\}$ with $i-j\ge k$, we have
\begin{equation}
	{a_{j}} \le q^{-k} a_{i}.
	\label{eq:lacunary-condition-k-steps}
\end{equation}

We now consider a sequence of natural numbers $(t_n)_{n\in\NN}$. Recall that $f(x)=\sum_{d=-D}^D c_de^{2\pi idx}$, $X_j=f(a_jU)$,  $S_n^f = \sum_{j=1}^n X_j$, and $T_n^f = \sum_{j=1}^n Y_j$, where $Y_1,Y_2,\ldots$ are independent copies of $X_1$. Then for any $r\in \NN$, we have
	\begin{align}\label{eq:expectation-solutions}
		\EE \prod_{j=1}^r X_{t_j} 
		& = \int_0^1 \prod_{j=1}^r \sum_{d=-D}^D c_d e^{2\pi ida_{t_j}x} \, \text{d}x
		\\ \nonumber
		& = \sum_{d_1,\ldots, d_r \in \{-D,\ldots , D\}} c_{d_1}\cdots c_{d_r} \int_0^1 e^{2\pi ix \sum\limits_{j=1}^r d_j a_{t_j}}  \, \text{d}x 
		\\ \nonumber
		& = \sum_{d_1,\ldots, d_r \in \{-D,\ldots , D\}} c_{d_1}\cdots c_{d_r} \ind_{\{\sum\limits_{j=1}^r d_j a_{t_j} =0 \}}.
	\end{align}
	Moreover, for every $n,m\in \NN$,
	\begin{align}	\label{eq:moms-Sn}
		\EE (S_n^f)^m 
		& = \int_0^1 \Bigl( \sum_{\ell=1}^n \sum_{d=-D}^D c_de^{2\pi i da_\ell x} \Bigr)^m  \, \text{d}x 
		\\ \nonumber
		 &= \sum_{\substack{\ell_1,\ldots, \ell_m \in \{1,\ldots ,n\} \\ d_1,\ldots, d_m \in \{-D,\ldots,D\} }} c_{d_1}\cdots c_{d_m} \int_0^1 e^{2\pi i x \sum\limits_{j=1}^m d_j a_{\ell_j}} \, \text{d}x
		\cr \nonumber
		 & = \sum_{\substack{\ell_1,\ldots, \ell_m \in \{1,\ldots ,n\} \\ d_1,\ldots, d_m \in \{-D,\ldots,D\} }} c_{d_1}\cdots c_{d_m} \ind_{\{ \sum\limits_{j=1}^m d_j a_{\ell_j} = 0\}},
	\end{align}
	and
	\begin{align}	\label{eq:moms-Sn-indep}
		\EE (T_n^f)^m 
		& = \int_{[0,1]^n} \Bigl( \sum_{\ell=1}^n \sum_{d=-D}^D c_de^{2\pi i da_\ell x_\ell} \Bigr)^m  \text{d}(x_1,\ldots,x_n) 
		\\ 
		 &= \sum_{\substack{\ell_1,\ldots, \ell_m \in \{1,\ldots ,n\} \\ d_1,\ldots, d_m \in \{-D,\ldots,D\} }} c_{d_1}\cdots c_{d_m} \int_{[0,1]^n} e^{2\pi i  \sum\limits_{j=1}^m d_j a_{\ell_j}x_{\ell_j}} \text{d}(x_1,\ldots,x_n) 
		\cr \nonumber
		 & = \sum_{\substack{\ell_1,\ldots, \ell_m \in \{1,\ldots ,n\} \\ d_1,\ldots, d_m \in \{-D,\ldots,D\} }} c_{d_1}\cdots c_{d_m} \prod_{k=1}^n \Biggl( \int_0^1 e^{2\pi i x a_k \sum\limits_{j=1}^m d_j \ind_{\{\ell_j=k\}}} \text{d}x \Biggr)
		 \cr
		 & = \sum_{\substack{\ell_1,\ldots, \ell_m \in \{1,\ldots ,n\} \\ d_1,\ldots, d_m \in \{-D,\ldots,D\} }} c_{d_1}\cdots c_{d_m} \prod_{k=1}^n	  \ind_{\{ \sum\limits_{j=1}^m d_j \ind_{\{\ell_j=k\}} = 0\}}.
	\end{align}

	In general, when condition \eqref{condition-var} on the variance of $S_n^f$ is violated, we cannot hope that $(\frac{S_n}{y_n})_{n\in\NN}$ satisfies an MDP for any sequence $(y_n)_n$ going to $\infty$, as the example of telescoping sums described in the introduction shows.
	Let us present three natural conditions implying  \eqref{condition-var}.
	
	\begin{proposition}	\label{prop:conds-implying-var}
		Assume that $(a_n)_{n\in\NN}$ is a (not necessarily lacunary) sequence of positive integers, and $f\not\equiv 0$ is a trigonometric polynomial of the form \eqref{eq:trigonometric polynomial} with degree $D$, $c_0=0$, and  $c_d=-c_d$ for every $d$.
		Then each of the following conditions implies   \eqref{condition-var}:
		\begin{enumerate}[(A)]
			\item \label{item:cond-c_d-nonnegative}
			$c_d\ge 0$ for every $d\in \{1,\ldots , D\}$.
			\item	\label{item:cond-number-of-solutions}
			 \[\sup_{d, d' \in \{1,\ldots , D\} } \# \Bigl\{ (\ell,\ell') \in \{1,\ldots, n\} ^2 \colon da_\ell=d'a_{\ell'}, \  \ell \neq \ell' \Bigr\} = o(n).\]
			\item	\label{item:cond-a_n-bigger-D}
			There exists $k_0\in \NN$ such that
			 $\frac{a_{n+1}}{a_n} > D$ for every integer $n \ge k_0$.
		\end{enumerate}
	\end{proposition}
	We postpone the proof of Proposition~\ref{prop:conds-implying-var} to the end of this section.
	The proof that \ref{item:cond-number-of-solutions} implies \eqref{condition-var} is  contained on page 285 of \cite{AB2010}, where the authors noticed\footnote{In a remark following \cite[Theorem~1.2]{AB2010}.} that condition \ref{item:cond-number-of-solutions} holds under the large gap condition. 
	Since our proof needs no additional notation and is short, we decided to include it in this article.
	Moreover, in the last part of the proof of Theorem~\ref{thm:MDP-large-gap} we show that under the large gap condition property \ref{item:cond-number-of-solutions} holds even with $O(1)$ in the place of $o(n)$.

	\begin{example}\label{ex:cosine}
		In a classical case where $f(x) = \cos(2\pi x)$, we may easily calculate $\EE S_n^2$ using, say, the formula presented in \eqref{eq:moms-Sn}:
		\[
			\EE S_n^2 = \frac 14 \# \Bigl\{ (\eps, \eps', \ell,\ell')\in \{-1,1\}^n\times \{1,\ldots, n\}^2 \colon \eps a_\ell +\eps'a_{\ell'}=0  \Bigr\}= \frac n2,
		\]
		so condition \eqref{condition-var} holds and Theorem~\ref{thm:integer-ratios} implies that for every $(a_n)_n$ and $(z_n)_n$ satisfying the assumptions of Theorem~\ref{thm:integer-ratios} or Theorem~\ref{thm:MDP-large-gap}, a random sequence $(\frac{S_n}{y_n})_n$ satisfies a moderate deviation principle with speed $x_n$ and rate function $I(t)=t^2$, where $x_n = nz_n$ and $y_n = n\sqrt{z_n}$. 
	\end{example}

	\begin{remark}	\label{rmk:var-upper-bound}
	Let us note that the bound reverse to \eqref{condition-var} is always satisfied.
	More precisely,
	assume that $(a_n)_n$ is a (not necessarily lacunary)  sequence of positive integers  and $f$ is  a trigonometric polynomial with $c_d=-c_d$ for every $d$. 
	Then
	\begin{equation} 	\label{eq:condition-var-reversed}
		\EE (S_n^f)^2 \le C(f) n,
	\end{equation}
	where $C(f)\in(0,\infty)$ is a constant depending only on $f$. 
	
	Indeed, \eqref{eq:moms-Sn} and inequality $2xy\le (x^2 +y^2)$ imply that
	\begin{multline*}
		\EE S_n^2 
		= 2n \sum_{d=0}^D c_d^2 
		+ 2 \sum_{\substack { d,d'\in \{0,\ldots, D \} \\ d\neq d'}} c_d c_{d'} \# \Bigl\{ (\ell,\ell') \in \{1,\ldots , n \} \colon da_\ell = d'a_{\ell'} \Bigr\}
		\\ \le 2n \sum_{d=0}^D c_d^2 
		+  \sum_{d,d'\in \{0,\ldots, D \} } ( c_d^2  + c_{d'}^2 ) n 
		=  4n \sum_{d=0}^D c_d^2 .
	\end{multline*}
	\end{remark}

	Before we move on to the proof of Theorem~\ref{thm:integer-ratios} we need some preparation. We first introduce a modification of a standard definition of the dependency graph. Since in our case there is a lack of independence, we substitute it with a type of uncorrelation property. This leads us to the following definition.

\begin{definition}
	A graph $G$ with the vertex set $V$ is called a \emph{correlation graph of range $M$} for the family of random variables $(Z_v)_{v\in V}$ if the following property is satisfied: {\upshape if $V_1$ and $V_2$ are two disjoint nonempty subsets of vertices such  that there are no edges in $G$ between $V_1$ and $V_2$, then for every sequences $(v_n)_{n\in\NN}$ of elements of $V_1$ and $(w_n)_{n\in\NN}$ of elements of $V_2$  and for every $r,s\in \NN$ such that $r+s\le M$,
	\begin{equation}
		\EE \prod_{i=1}^r Z_{v_i} \prod_{j=1}^s Z_{w_j} = \EE \prod_{i=1}^r Z_{v_i} \ \EE\prod_{j=1}^s Z_{w_j} .
		\label{eq:uncorr-condition}
	\end{equation}}
\end{definition}

Let us recall that for $m\in\NN$ the $m$-th cumulant of a random variable $Z$ satisfying $\EE e^{tZ}<\infty$ in a neighbourhood of $0$ is defined as
  \[
    \gamma_m(Z) = \frac{\partial^m}{\partial t^m} \Big(\log \EE\big(e^{tZ}\big) \Big)\Big|_{t=0}.
  \]
We shall use the following version of \cite[Theorem~9.1.7]{Mod-phi-book} in order to prove Theorem~\ref{thm:integer-ratios}.

\begin{theorem}\label{thm:corr-graphs}
	Let $A>0$, $M \in \NN$, and let $G$ be a correlation graph of range $M\in\NN$ with vertex set $V$ for the family $(Z_v)_{v\in V}$ of random variables almost surely valued in $[-A,A]$. 
	Let $W\subseteq V$ be of size $n\in\NN$ and let $\deg(G)$ be the maximal degree of vertices in $G$.
	Then for every integer $m\le M$,
	    \[
		\Bigl| \gamma_m \bigl(\sum_{v\in W} Z_v\bigr) \Bigr| \le 2(2m)^{m-2} n\bigl(\deg(G)+1\bigr)^{m-1}A^m.
	    \]
\end{theorem}
Its proof is exactly the same as the proof of \cite[Theorem~9.1.7]{Mod-phi-book} --- the only step in this proof using the independence assumption is \cite[condition~(9.7)]{Mod-phi-book}, and this condition is exactly the uncorrelation assumption \eqref{eq:uncorr-condition} above. Moreover, to estimate the cumulants up to order $M$, we only need the information about mixed moments up to order $M$, so we need to assume \eqref{eq:uncorr-condition} only for $r$ and $s$ satisfying $r+s\le M$.

In order to prove Theorem~\ref{thm:integer-ratios} we shall use a bound on cumulants obtained by Theorem~\ref{thm:corr-graphs}. Therefore, we require the following proposition (which is valid even if condition \eqref{condition-var} is not satisfied). It is  a key step in the proof of Theorem~\ref{thm:integer-ratios}. Recall that if $v$ and $w$ are vertices in an undirected graph $G$, then $v\sim w$ means that there is an edge in $G$ between $v$ and $w$.
	
	\begin{proposition}\label{prop:graph-construction}
		Let $(a_n)_{n\in\NN}$ be a lacunary sequence satisfying $\frac{a_{n+1}}{a_n}\in \{2,3,\ldots\}$ for every $n\in\NN$ and let $X_n=f(Ua_n)$, $n\in\NN$, where $f$ is the centered and symmetric trigonometric polynomial of degree $D\in\NN$ from \eqref{eq:trigonometric polynomial}. 
		Assume that $M\in\NN$
		and set $k= \lceil \log (MD)/ \log 2\rceil$.
		Let $G$ be the graph with the vertex set $V=\NN$ such that 
		\[
		i\sim j \text{\qquad if and only if \quad} |i-j|\le k.
		\] 
		Then $G$ is a correlation graph of range $M$ for the family $(X_n)_{n\in\NN}$.
	\end{proposition}

	\begin{proof}
		Fix two nonempty disjoint subsets of vertices $V_1\subseteq \NN$ and $V_2\subseteq \NN$  such  that there are no edges in $G$ between $V_1$ and $V_2$. 
		Fix a sequence  $(v_n)_{n\in\NN}$  of elements of $V_1$, 
		 a sequence $(w_n)_{n\in\NN}$ of elements of $V_2$, and two numbers $r,s\in\NN$ such that $r+s\le M$.

		By \eqref{eq:expectation-solutions} we have
		\[
			\EE \prod_{i=1}^r X_{v_i} \prod_{j=1}^s X_{w_j} = \sum_{\substack{\eps_1,\ldots, \eps_r \in \{-D,\ldots , D\} \\ \eta_1,\ldots, \eta_s \in \{-D,\ldots, D\} }} c_{\eps_1}\cdots c_{\eps_r}c_{\eta_1}\cdots c_{\eta_s} \ind_{\{\sum_{i=1}^r \eps_i a_{v_i} + \sum_{j=1}^s \eta_j a_{w_j}=0 \}}
		\]
		and
		\begin{multline*}
			\EE \prod_{i=1}^r X_{v_i} \EE\prod_{j=1}^s X_{w_j} 
			\\
			= \sum_{\substack{\eps_1,\ldots, \eps_r \in \{-D,\ldots , D\} \\ \eta_1,\ldots, \eta_s \in \{-D,\ldots, D\} }} c_{\eps_1}\cdots c_{\eps_r}c_{\eta_1}\cdots c_{\eta_s} \ind_{\{\sum_{i=1}^r \eps_i a_{v_i} =0 \}} \ind_{\{  \sum_{j=1}^s\eta_j a_{w_j}=0 \}}.
		\end{multline*}
		Moreover, every pair of solutions (in unknown variables $\eps_i,\eta_j\in\{-D,\ldots,D\}$, $i\le r$, $j\le s$) to the equations $\sum_{i=1}^r \eps_i a_{v_i}=0$ and  $\sum_{j=1}^s \eta_j a_{w_j}=0$ provide a solution to the equation $\sum_{i=1}^r \eps_i a_{v_i}+\sum_{j=1}^s \eta_j a_{w_j}=0$. Hence, in order to confirm \eqref{eq:uncorr-condition}, we only need to prove the following claim: every solution (in unknown variables $\eps_i,\eta_j\in\{-D,\ldots, D\}$) to the equation
		\begin{equation}
			\sum_{i=1}^r \eps_i a_{v_i}+\sum_{j=1}^s \eta_j a_{w_j}=0
			\label{eq:split-me}
		\end{equation}
		splits into two solutions to the equations: $\sum_{i=1}^r \eps_i a_{v_i}=0$  and  $\sum_{j=1}^s \eta_j a_{w_j}=0$. To be more precise, we only need to show that if $\eps_i$'s and $\eta_j$'s satisfy \eqref{eq:split-me}, then 
		\begin{equation}
			\sum_{i=1}^r \eps_i a_{v_i}=0 \quad \text{and}\quad \sum_{j=1}^s \eta_j a_{w_j}=0
			\label{eq:i-am-split}.
		\end{equation}
		
		Fix the solution to \eqref{eq:split-me} from the set $\{-D, \ldots,D\}^{r+s}$, and assume by contradiction that both $\sum_{i=1}^r \eps_i a_{v_i}\neq 0$ and  $\sum_{j=1}^s \eta_j a_{w_j}\neq0$. Let $L\in \{2,3,\ldots\}$, and let $(C_1, \ldots, C_L)$ be a sequence of subsets of $\NN$ satisfying the following properties:	
		\begin{enumerate}
			\item[(1)] $C_1, \ldots, C_L$ are mutually disjoint, non-empty, contain only elements of $\{v_i\colon i\le r\}\cup\{w_j\colon j\le s\}$, and every element of $\{v_i\colon i\le r\}\cup\{w_j\colon j\le s\}$ belongs to $C_1\cup \cdots \cup C_L$.
			\item[(2)] For every $\ell\le L$ the set $C_\ell$ contains either only elements of  $\{v_i\colon i\le r\}$ or only the elements of $\{w_j\colon j\le s\}$.
			\item[(3)] If $\ell\le L-1$ and $C_\ell$ contains only elements of $\{v_i\colon i\le r\}$, then $C_{\ell+1}$ contains only elements of $\{w_j\colon j\le s\}$.
			\item[(3')] If $\ell\le L-1$ and $C_\ell$ contains only elements of $\{w_j\colon j\le s\}$, then $C_{\ell+1}$ contains only elements of $\{v_i\colon i\le r\}$.
			\item[(4)] If $\ell\le L-1$, then the maximal element of $C_\ell$ is smaller than the minimal element of $C_{\ell+1}$.
		\end{enumerate}
		It is clear that there exist exactly one $L\ge 2$ and exactly one sequence $(C_1, \ldots, C_L)$ satisfying these properties.
Property (1) implies that 
		\[
			0 \neq  \sum_{i=1}^r \eps_i a_{v_i} = \sum_{\ell=1}^L \sum_{i=1}^r \eps_i a_{v_i}\ind_{\{v_i\in C_\ell\}} 
		\]
		and similarly
		\[
			0 \neq  \sum_{j=1}^s \eta_j a_{w_j} = \sum_{\ell=1}^L  \sum_{j=1}^s \eta_j a_{w_j}\ind_{\{w_j\in C_\ell\}} .
		\]
		Therefore,  there exists $\ell\le L$ such that 
		\[
			\sum_{i=1}^r \eps_i a_{v_i}\ind_{\{v_i\in C_\ell\}}  \neq 0 \quad 
 		\text{or} \quad
			 \sum_{j=1}^s \eta_j a_{w_j}\ind_{\{w_j\in C_\ell\}} \neq 0;
		\]
		note that only one of the sums above may be nonzero, since by property (2) one of them has all the indicators equal to $0$. Let $\ell_0$ be the maximal $\ell\le L$ with the above property and
		assume without loss of generality that 
		\[
			 \sum_{j=1}^s \eta_j a_{w_j}\ind_{\{w_j\in C_{\ell_0}\}} \neq 0, \qquad 
			  \sum_{i=1}^r \eps_i a_{v_i}\ind_{\{v_i\in C_{\ell_0}\}}  = 0,
		\]
		and the set $C_{\ell_0}$ contains only elements of $\{w_j \colon j\le s\}$.
		Let $j_0\le s$ be such that $w_{j_0} = \min C_{\ell_0}$.
		 
		By the maximality property of $\ell_0$, we know that for every $\ell>\ell_0$, we have
		\begin{equation}\label{eq:terms are zero for l larger than l0}
			\sum _{j=1}^s \eta_j a_{w_j}\ind_{\{w_j\in C_{\ell}\}} = 0 = \sum_{i=1}^r \eps_i a_{v_i}\ind_{\{v_i\in C_\ell\}}.
		\end{equation}
		Therefore, 
		\begin{equation}
			\sum_{\ell=\ell_0}^L \sum_{i=1}^r \eps_i a_{v_i}\ind_{\{v_i\in C_\ell\}}  +\sum_{\ell=\ell_0}^L\sum_{j=1}^s \eta_j a_{w_j}\ind_{\{w_j\in C_{\ell}\}} =\sum_{j=1}^s \eta_j a_{w_j}\ind_{\{w_j\in C_{\ell_0}\}}  \neq0.
			\label{eq:non-zero-div}
		\end{equation}
		In particular, the previous identity \eqref{eq:non-zero-div} together with \eqref{eq:split-me} yields that $\ell_0>1$. Thus, we may consider the set $C_{\ell_0-1}$ later on. Property (4) implies that every element of $\bigcup_{\ell=\ell_0}^L C_{\ell}$ is greater than or equal to $w_{j_0} = \min C_{\ell_0}$. Moreover, by our assumption on the gap sequence $(a_n)_{n\in\NN}$, $a_{n}$ divides $a_{n+1}$ for every $n\in \NN$. Therefore, every term on the left-hand side of \eqref{eq:non-zero-div} is divisible by $a_{w_{j_0}}$, and so is the whole sum on the left-hand side. Since this sum is non-zero, we must have
		\begin{equation}
			\biggl| \sum_{\ell=\ell_0}^L \sum_{i=1}^r \eps_i a_{v_i}\ind_{\{v_i\in C_\ell\}}  +\sum_{\ell=\ell_0}^L\sum_{j=1}^s \eta_j a_{w_j}\ind_{\{w_j\in C_{\ell}\}} \biggr| \ge a_{w_{j_0}};
			\label{eq:div-implies-ineq}
		\end{equation}
note that the previous estimate does not hold for general lacunary sequences without our assumption of integer ratios.	
	
	Since $C_{\ell_0}$ is a subset of $\{w_j \colon j\le s\}$, properties (2) and (3') imply that $C_{\ell_0-1}$ is a subset of $\{v_i \colon i\le r\}$, and so there exists $i_0\le r$  such that $v_{i_0} = \max C_{\ell_0-1}$. 
It now follows from property (4) that
		\[
			v_{i_0}=\max \bigcup_{\ell=1}^{\ell_0-1} C_\ell = \max C_{\ell_0-1}<\min C_{\ell_0}=w_{j_0}.
		\] 
Since there are no edges between the vertex sets $V_1$ and $V_2$, we have $v_{i_0}\not\sim w_{j_0}$ and so by the definition of our graph $G$, 
	        \[
	           |v_{i_0}-w_{j_0}|\ge k+1.
	        \]  
Since $v_{i_0}<w_{j_0}$, we thus have $\max \bigcup_{\ell=1}^{\ell_0-1} C_\ell = v_{i_0} \le  w_{j_0}-k-1$. Hence, \eqref{eq:lacunary-condition-k-steps} implies that for every element $u\in\bigcup_{\ell=1}^{\ell_0-1}C_\ell$, we have 
                 \[ 
                    a_u \le 2^{-k-1}a_{w_{j_0}}.
                  \]   
Moreover, we trivially have the bounds $|\eps_i|, |\eta_j|\le D$ for all $i\le r$ and $j\le s$.
Therefore, the triangle inequality yields
		\begin{align*}
			 \biggl| \sum_{\ell=1}^{\ell_0-1}\sum_{i=1}^r \eps_i a_{v_i} &  \ind_{\{v_i\in C_\ell\}}  + \sum_{\ell=1}^{\ell_0-1} \sum_{j=1}^s \eta_j a_{w_j}\ind_{\{w_j\in C_{\ell}\}}  \biggr|  \cr
			& \le  D \sum_{\ell=1}^{\ell_0-1}\sum_{i=1}^r a_{v_i} \ind_{\{v_i\in C_\ell\}}   + D\sum_{\ell=1}^{\ell_0-1} \sum_{j=1}^s a_{w_j}\ind_{\{w_j\in C_{\ell}\}}  \cr
			& \le  D\sum_{\ell=1}^{\ell_0-1}\sum_{i=1}^r 2^{-k-1}a_{w_{j_0}}  \ind_{\{v_i\in C_\ell\}}  +D\sum_{\ell=1}^{\ell_0-1} \sum_{j=1}^s 2^{-k-1}a_{w_{j_0}} \ind_{\{w_j\in C_{\ell}\}}  \cr
			& \le  D2^{-k-1}a_{w_{j_0}} (s+r) 
			 \le  MD 2^{-k-1}a_{w_{j_0}} 
			 <  a_{w_{j_0}},
		\end{align*}
where the last bound follows directly from our choice $k= \lceil \log (MD)/ \log 2\rceil$.	
	This bound, the triangle inequality, and  estimate  \eqref{eq:div-implies-ineq} imply 		
	       \begin{align*}
		        \biggl| \sum_{i=1}^r \eps_i a_{v_i}+\sum_{j=1}^s \eta_j a_{w_j} \biggr| &=  \biggl| \sum_{\ell=1}^{L}\sum_{i=1}^r \eps_i a_{v_i}\ind_{\{v_i\in C_\ell\}}  + \sum_{\ell=1}^{L} \sum_{j=1}^s \eta_j a_{w_j}\ind_{\{w_j\in C_{\ell}\}}  \biggr| \cr
			&\ge \biggl| \sum_{\ell=\ell_0}^{L}\sum_{i=1}^r \eps_i a_{v_i}\ind_{\{v_i\in C_\ell\}}  + \sum_{\ell=\ell_0}^L \sum_{j=1}^s \eta_j a_{w_j}\ind_{\{w_j\in C_{\ell}\}}  \biggr| \cr
			& \qquad-  \biggl| \sum_{l=1}^{l_0-1}\sum_{i=1}^r \eps_i a_{v_i}\ind_{\{v_i\in C_l\}}  + \sum_{l=1}^{l_0-1} \sum_{j=1}^s \eta_j a_{w_j}\ind_{\{w_j\in C_{l}\}}  \biggr|
			>  0,
		\end{align*}
		which contradicts \eqref{eq:split-me}. The proof is therefore complete.
	\end{proof}

	The following result by Rudzkis, Saulis, and Statuljavi\v{c}us from \cite{RSS78} (see also \cite[Lemma~2.2]{SS-book}) is a classical tool which allows to deduce the MDP from  bounds on cumulants.
	
	\begin{lemma}[{\cite{RSS78} or \cite[Lemma~2.2]{SS-book}}]	\label{lem:cum-tails}
	Let $\Theta >0$ and $1\le s\le 2\Theta^2$. Assume that $Z$ is a centered random variable with variance  $1$ satisfying the condition
	\begin{equation*}
		|\gamma_m(Z)| \le \frac{(m-2)!}{\Theta^{m-2}} \qquad \text{for all } m=3, 4, \ldots, s+2.
	\end{equation*}
	Then for every $0\le x < \frac{\sqrt s }{ 3\sqrt e}$ we have the tail bound
	\begin{equation} \label{eq:tail-bound1}
		\PP(Z>x) = \PP(G>x) e^{L(x)} \Bigl( 1+ f(x) \frac{x+1}{\sqrt s} \Bigr),
	\end{equation}
		where $G\sim \mathcal{N}(0,1)$, $f\colon [0,  \frac{\sqrt s }{3\sqrt e} ) \to (0,\infty) $ is a function satisfying
		\[
			f(x) \le \frac{117 +96 s \exp\bigl(-\frac 12 s^{1/4}(1- \frac{3\sqrt e x}{\sqrt s})\bigr)}{1- \frac{3\sqrt e x }{\sqrt s} },
		\]  
		and $|L(x)| \le \frac{5|x|^3}{  4\Theta}$ for every  $|x|<\frac{ \sqrt{2} \Theta}{3\sqrt e}$. 
	\end{lemma}

	We shall now use Lemma~\ref{lem:cum-tails} to prove the following proposition in the spirit of \cite[Theorem~1.1]{DE13} by D{\"o}ring and Eichelsbacher. This is the last tool we need in the proof of Theorem~\ref{thm:integer-ratios}.
	\begin{proposition}	\label{prop:cum-MDP}
		Let $\beta>0$, $\gamma \ge0$, and $ 0<\Delta_n \le n^\beta$ for  sufficiently large $n\in \NN$.
		Assume that $(b_n)_{n\in\NN}$ is an increasing sequence of positive numbers converging to infinity slow enough to ensure that
	\begin{equation*}
	\frac{b_n  (\log n)^\gamma}{\Delta_n} \mathop{\longrightarrow}^{n\to \infty} 0.
	\end{equation*}
	 Let $(S_n)_{n\in\NN}$ be a sequence of centered random variables with finite moments of all orders and $\sigma_n= (\EE S_n^2)^{1/2}$. If 
	 \begin{equation}
	|\gamma_m(S_n)|\le \frac{m^{m-2} (\log m)^{\gamma m}\sigma_n^m}{\Delta_n^{m-2}}
	\label{eq:cond-DE}
	\end{equation}
	holds for all $m\in \{3,4,\ldots\}$ and all sufficiently large $n\in \NN$, then $(\frac{S_n}{\sigma_nb_n})_{n\in\NN}$ satisfies an MDP with speed $b_n^2$ and rate function $I(t)=t^2/2$.
	\end{proposition}
	\begin{proof}
		Since the $m$-th cumulant is $m$-homogeneous, $\gamma_m(S_n \sigma_n^{-1}) = \sigma_n^{-m}\gamma_m(S_n)$,  so we may and will assume that $\sigma_n =1$ for all $n\in\NN$. Let $\Theta_n = \frac{ \Delta_n }{3e^{1+2\gamma}(2\beta\log n)^\gamma} $ and $s_n= 2\Theta_n^2$.
	We shall consider only $n\ge 2$, so $\Theta_n$ is finite.
	
	Since $\frac{b_n  (\log n)^\gamma}{\Delta_n} \to 0$ and $b_n\to \infty$ as $n\to \infty$, we have for sufficiently large $n\in\NN$ that $3e^{1+2\gamma}(2\beta\log n)^\gamma<\Delta_n$.	
Now we combine assumption \eqref{eq:cond-DE} with the loose Stirling bound $k! e^k>k^k$ to show that for all $m\in\{3, 4, \ldots, n^{2\beta} \}$ and sufficiently large $n\in \NN$,
	\begin{align*}
			|\gamma_m(S_n)| 
			& \leq   \frac{m^{m-2} (\log m)^{\gamma m}}{\Delta_n^{m-2}} 
			 \leq  \frac{m^{m-2} (\log m)^{2 \gamma} (2\beta \log n)^{\gamma (m-2)}}{\Delta_n^{m-2}} \cr
			& \le \frac{ e^{m-2} (m/(m-2))^{m-2} (m-2)! (\log m)^{2\gamma}  (2\beta \log n)^{\gamma (m-2)}}{\Delta_n^{m-2}}  \cr
			 & \le \frac{(m-2)! (3e^{1+2\gamma})^{m-2} (2\beta\log n)^{\gamma (m-2)}}{\Delta_n^{m-2}}
			  = \frac{(m-2)!}{\Theta^{m-2}}.
	\end{align*}
Since we have $3e^{1+2\gamma}(2\beta\log n)^\gamma<\Delta_n \le n^\beta$ for sufficiently large $n$,  $s_n>2$ and  $\Theta_n \leq \frac{n^\beta}{2}$, so  $s_n=2\Theta_n^2 \leq \frac{n^{2\beta}}{2}$, which implies 
		\[
		  s_n+2 < 2s_n \leq n^{2\beta};
		\]
this means that our upper bound on $|\gamma_m(S_n)|$ holds for all $m\in\{3,4,\ldots, s_n+2\}$.		
Therefore, the assumptions of Lemma~\ref{lem:cum-tails} are satisfied with  $Z=S_n$, $s=s_n$, and $\Theta=\Theta_n$. 
Moreover, since $b_n  (\log n)^\gamma/\Delta_n \rightarrow 0$, we have that
$b_n / \Theta_n \to 0$. Recall also our assumption that $b_n\to \infty$. Thus, for every sufficiently large $n\in\NN$ and every $0\le x\le \frac{\Theta_n}{10 b_n}<0.9 \frac{\sqrt{s_n}}{3\sqrt eb_n}$  we have 
		\begin{align*}
			\frac{1}{b_n^2} \log \PP \biggl( \frac{S_n}{b_n} > x \biggr) 
			& = \frac{1}{b_n^2} \log \PP \biggl(\frac{G}{b_n} > x \biggr) 
			+  \frac{1}{b_n^2} \biggl[ L(xb_n) + \log\Bigl (1+ f(xb_n)\frac{xb_n+1} {\sqrt{s_n}} \Bigr) \biggr]
			\\ & =- \frac{x^2}2 +  \log\bigl(xb_n+2\bigr)O\Bigl(\frac1{b_n^2}\Bigr)+ \frac{1}{b_n^2} \biggl[ x^3 O\biggl(\frac{b_n^3}{\Theta_n}\biggr) +O(1)  \biggr]
			\\ & = -\frac{x^2}2 +o(1)\Bigl( \log\bigl(x+2\bigr)+x^3+1\Bigr),
		\end{align*}
		where in the second equality we used the bound 
		\[
			\log \PP \bigl(G> y \bigr) = - \frac{y^2}2 + \log\bigl(y+2\bigr)O(1) \qquad \text{for every }y>0,
		\]
		 which follows from the Komatsu inequality \cite{Komatsu}. Moreover, the sequence $( \frac{\Theta_n}{10 b_n})_{n\in\NN}$ of elements dominating $x$ goes to $\infty$ when $n\to \infty$. Hence, for all $x\ge 0$, we obtain
		\[
			\frac{1}{b_n^2} \log \PP \biggl(  \frac{S_n}{b_n} > x \biggr) \mathop{\longrightarrow}^{n\to \infty} -\frac{x^2}2.
		\]
		Since $(-S_n)_{n\in\NN}$ satisfies the same assumptions as $(S_n)_{n\in\NN}$, we also have that for every $x\ge0$, 
		\[
			\frac{1}{b_n^2} \log \PP \biggl( \frac{S_n}{b_n} < -x \biggr)
			= \frac{1}{b_n^2} \log \PP \biggl( - \frac{S_n}{b_n} > x \biggr) \mathop{\longrightarrow}^{n\to \infty} -\frac{x^2}2.
		\]
		Thus, for every interval $J\subseteq \RR$,
		\[
			\frac{1}{b_n^2} \log \PP \biggl( \frac{S_n}{b_n} \in J \biggr)
			 \mathop{\longrightarrow}^{n\to \infty}  - \inf_{x\in J} \frac{x^2}2.
		\]
		This means that $(S_n/b_n)_{n\in\NN}$ satisfies an MDP with speed $b_n^2$ and rate function $I(t)=t^2/2$ --- this classical fact follows from \cite[Theorem~4.1.11 and Lemma~1.2.18]{DZ-book}, applied to the family of measures 
		\[
			\mu_\eps(B) \coloneqq \PP(\sqrt \eps S_{n(\eps)} \in B),
		\] 
		where $n(\eps)\coloneqq \max \{n\in \NN \colon b_n^2\le \eps^{-1} \}$. This completes the proof.
	\end{proof}
	
	We are now ready to present the proof of Theorem~\ref{thm:integer-ratios}.
	
	\begin{proof}[Proof of Theorem~\ref{thm:integer-ratios}]
Let $S_n=S_n^f$.
Since $f$ is a trigonometric polynomial, it is bounded, so there exists $A>0$ such that for every $n\in \NN$, $|X_n|=|f(a_nU)|\le A$ a.s. 
For an arbitrary $n\in\NN$, we apply Theorem~\ref{thm:corr-graphs} with  $Z_v=X_i$, $W=\{1,\ldots, n\}$, and the correlation graph described in Proposition~\ref{prop:graph-construction} to obtain that there exists a constant $C=C(A,D)\ge 1$  such that, for every   $m\ge 3$ and  $n \in \NN$,
	\begin{align*}
		|\gamma_m (S_n)|
		& \le 2(2m)^{m-2} n (2\log (mD) +3)^{m-1}A^m 
		 \le   m^{m-2}  C^{m-2} (\log m)^{m} n \cr
		& \hspace{-3pt} \mathop{\le}^{\eqref{condition-var}} \frac{m^{m-2} (\log m)^m \bigl(\sqrt{ \EE S_n^2}\bigr)^m}{ \bigl(C^{-1} \sqrt {\EE S_n^2}\bigr)^{m-2} \min\{1,c(f,a)\}} 
		 \le \frac{m^{m-2} (\log m)^m \bigl(\sqrt{ \EE S_n^2}\bigr)^m}{ \bigl(C^{-1} \min\{1,c(f,a)\} \sqrt{ \EE S_n^2}\bigr)^{m-2} } .
	\end{align*}
Therefore, the cumulant bound \eqref{eq:cond-DE} of Proposition \ref{prop:cum-MDP} is satisfied with $\gamma=1$ and $\Delta_n=\frac{\min\{1,c(f,a)\}\sqrt{\EE S_n^2}}{C}$.
It follows  from Remark~\ref{rmk:var-upper-bound} (which shows that the reverse bound to \eqref{condition-var} always holds) that  $\Delta_n=O(\sqrt n)$, so $\Delta_n\le n$ for sufficiently large $n$.
Now we choose $b_n=\sqrt { z_n \EE S_n^2}$, $n\in\NN$. 
Then, using assumptions \eqref{condition-var} and $z_nn\to\infty$ as $n\to\infty$, we obtain 
        \[
           b_n \ge \sqrt{c(f,a) n z_n}  \ \mathop{\longrightarrow}^{n\to \infty} \infty.
        \]   
Moreover, because we assume $z_n(\log n)^2 \to 0$ as $n\to\infty$, we also have
	\[
		\frac{b_n \log n }{\Delta_n} 
		=\sqrt{ z_n} \log n\frac{ C}{ \min\{1,c(f,a)\}} \ \mathop{\longrightarrow}^{n\to \infty} 0.
	\]
Hence, Proposition~\ref{prop:cum-MDP} applied with $\beta=1=\gamma$ yields that the sequence 
	 $
	   (\frac{S_n}{ b_n\sqrt{\EE S_n^2}})_{n\in\NN}=(\frac{S_n}{ \sqrt{z_n} \EE S_n^2})_{n\in\NN}=(\frac{S_n}{ y_n})_{n\in\NN}
	 $
satisfies an MDP  with speed $b_n^2= z_n \EE S_n^2=x_n$ and rate function $I(t)=t^2/2$. This completes the proof.
\end{proof}
	
	\begin{remark} 	\label{rmk:no-mod-G-for-free}
		Although Lemma~\ref{lem:cum-tails} delivers precise tail bounds similar to the mod-Gaussian convergence (see \cite[Section~5.3]{Mod-phi-book} for more details), it is not sufficient for providing mod-Gaussian convergence in the setting of Theorem~\ref{thm:integer-ratios}, since our bounds on cumulants are precise only up to logarithm, whereas in order to have mod-Gausssian convergence we should prove sharp bounds on cumulants, i.e., without a~logarithmic term in \eqref{eq:cond-DE}.
		This surely does not mean that the mod-Gausssian convergence does not hold in the setting of Theorem~\ref{thm:integer-ratios}; our methods simply do not work in this case.
	\end{remark}
	
	\begin{proof}[Proof of Proposition~\ref{prop:conds-implying-var}]
		Since $f\not\equiv 0$, we have $\frac 12 \EE X_1^2=\sum_{d=0}^D c_d^2 >0$.
		
		\vskip 2mm
		
		\textbf{ Assume that condition \ref{item:cond-c_d-nonnegative} holds.} Then identity \eqref{eq:moms-Sn} implies that 
		\begin{align*}
			\EE S_n^2 & = 2n \sum_{d=1}^D c_d^2 
		\hspace{1pt}+ \hspace{2pt}2 \hspace{-3pt} \sum_{\substack { d,d'\in \{1,\ldots, D \} \\ d\neq d'}} \hspace{-3pt} c_d c_{d'} \# \Bigl\{ (\ell,\ell') \in \{1,\ldots , n \}^2 \colon da_\ell = d'a_{\ell'} \Bigr\} \ge 2n \sum_{d=1}^D c_d^2,
		\end{align*}
		so \eqref{condition-var} is satisfied with a constant $c(f,a) = 2\sum_{d=0}^D c_d^2>0$ depending only on $f$.
		
		\textbf{ Assume that condition \ref{item:cond-number-of-solutions} holds.} Then  identity \eqref{eq:moms-Sn} implies that 
		\begin{align*}	
			\nonumber 
			\EE S_n^2 
			 & = 2n \sum_{d=1}^D  c_d^2 
		 + 2 \sum_{\substack { d,d'\in \{1,\ldots, D \} \\ d\neq d'}} c_d c_{d'} \# \Bigl\{ (\ell,\ell') \in \{1,\ldots , n \}^2 \colon da_\ell = d'a_{\ell'} \Bigr\}
		\\ \nonumber
		  & =  2n \sum_{d=1}^D  c_d^2 
		+ 2 \sum_{ \substack { d,d'\in \{1,\ldots, D \} \\ d\neq d'} } c_d c_{d'} \# \Bigl\{ (\ell,\ell') \in \{1,\ldots , n \}^2 \colon da_\ell = d'a_{\ell'}, \ \ell\neq \ell' \Bigr\}
		\end{align*}
		\begin{align}
		\label{eq:var-pom}
	 &\ge 2n \sum_{d=1}^D c_d^2 - 2D(D-1) \max_{1\le d\le D}c_d^2\, o(n) = n\EE X_1^2 - o(n) \ge c(f,a) n.
		\end{align}

		\textbf{ Assume that condition \ref{item:cond-a_n-bigger-D} holds.} 
		Then, for every $d,d' \in \{1,\ldots ,D\}$  and all integers $\ell \ge k_0$ and $\ell'< \ell$, we have $da_\ell > dD a_{\ell'} \ge d'a_{\ell'}$.
		Hence, for every $d,d' \in \{1,\ldots ,D\}$ we have
		\begin{multline*}
			\# \Bigl\{ (\ell,\ell') \in \{1,\ldots , n \}^2 \colon da_\ell = d'a_{\ell'}, \ l\neq l' \Bigr\} 
			\\ = \# \Bigl\{ (\ell,\ell') \in \{1,\ldots , k_0 - 1\}^2 \colon da_\ell = d'a_{\ell'} , \ \ell\neq \ell' \Bigr\}  \le (k_0-1)^2 = o(n).
		\end{multline*}
		Therefore, condition  \ref{item:cond-number-of-solutions} is satisfied and   the assertion follows from the previous paragraph. 
	\end{proof}
	
	We shall now present the proof of our second main result dealing with the case of  large gaps.
	
	\begin{proof}[Proof of Theorem~\ref{thm:MDP-large-gap}]
		Let us write $S_n=S_n^f$, $n\in\NN$.
		We begin with the proof of
		the MDP. By the G{\"a}rtner-Ellis theorem (see, e.g., \cite[Theorem 2.3.6]{DZ-book}) it suffices to prove that for every $\theta\in\RR$,
		\begin{equation}	\label{eq:aim-GE}
			 \lim_{n\to \infty } \frac{1}{nz_n} \log \EE e^{\theta  \sqrt{z_n} S_n} 
			 =\lim_{n\to \infty } \frac{\EE X_1^2}{x_n} \log \EE e^{\theta  x_n \frac{S_n}{y_n}} 
			 = \frac{\theta^2\EE X_1^2}2.
		\end{equation}
We shall show that the convergence is locally uniform on $\RR$. To this end, let us fix $T>0$ and consider $|\theta|\le T$. Every time we write $o(\cdot)$, $O(\cdot)$, this should be understood as uniform on the interval $[-T,T]$.
		
Since $\frac{a_{n+1}}{a_n} \to \infty$, there exists $k_0\in\NN$ such that  
		\begin{equation}\label{eq: estimate successive an}
		   a_{n+1}\ge 4Da_n \qquad \text{for every } n\ge k_0,
		\end{equation}
where $D\in\NN$ is the degree of the trigonometric polynomial $f$.	 
Thus, for every  $\widetilde{\ell}\in \NN$, $\widetilde{d}_1,\ldots, \widetilde{d}_{\widetilde{\ell}} \in\{-2D,\ldots, 2D\}\setminus\{0\}$, and a set of pairwise distinct integers  $ j_1, \ldots, j_{\widetilde{\ell}}> k_0$ (for which we denote $n_0\coloneqq j_{r_0}=\max_{1\leq r \leq \widetilde{\ell}} j_r$, with $r_0\in\{1,\dots, \widetilde{\ell}\}$), we have
		\begin{align*}
			|a_{j_{r_0}}  \widetilde{d}_{r_0}|
			& \ge a_{n_0} \stackrel{\eqref{eq: estimate successive an}}{\ge} 4D a_{n_0-1} 
			> 2D  \frac{a_{n_0-1}}{1-(4D)^{-1}} 
			> 2D   \sum_{i=0}^{n_0-1-k_0} \frac{a_{n_0-1}}{(4D)^i}
			\cr &
			\ge 2D \sum_{1\le r\le \widetilde{\ell}, r\neq r_0} \frac{a_{n_0-1}}{(4D)^{n_0-1-j_r}}
			\stackrel{\eqref{eq: estimate successive an}}{\ge} 2D \sum_{1\le r\le \widetilde{\ell}, r\neq r_0} a_{j_r}
			\cr &
			\ge  \sum_{1\le r\le \widetilde{\ell}, r\neq r_0} a_{j_r}  |\widetilde{d}_r| 
			\ge \Bigl| \sum_{1\le r\le \widetilde{\ell}, r\neq r_0} a_{j_r}  \widetilde{d}_r \Bigr|.
		\end{align*}
		Hence, there are no solutions to the equation $\sum_{r=1}^{\widetilde{\ell}} \widetilde{d}_r a_{j_r}=0$ with $\widetilde{d}_r$'s and $j_r$'s as above. This means that for a fixed $\ell\in\NN$ every solution to the equation $\sum_{r=1}^\ell d_r a_{j_r}=0$  in  integers $d_1,\dots,d_\ell \in \{-D,\ldots, D\}$ and $j_1,\ldots, j_\ell > k_0$ such that $\#\{r\le \ell\colon j_r=m\} \in \{0,1,2\}$ for every $m\in \NN$ must actually satisfy $\#\{r\le \ell\colon j_r=m\}\in \{0,2\}$ for every $m\in\NN$, and $d_{r}=-d_{p}$ whenever $j_r=j_p$ and $r\neq p$. 
		Thus, 
		formula \eqref{eq:expectation-solutions} implies that for every collection of $\delta_1,\ldots, \delta_\ell\in\{1,2\}$ and pairwise distinct integers $j_1,\ldots, j_\ell > k_0$,
		\begin{equation}
		\label{eq:MDP-large-gaps-pom1}
			\EE (X_{j_1}^{\delta_1} \cdots X_{j_\ell}^{\delta_\ell}) = \ind_{\{\delta_1=\ldots=\delta_\ell=2\}} \EE X_{j_1}^{\delta_1} \cdots \EE X_{j_\ell}^{\delta_\ell} = \ind_{\{\delta_1=\ldots=\delta_\ell=2\}} (\EE X_1^2)^\ell.
		\end{equation}
		Moreover \eqref{eq: estimate successive an} implies that for every $d,d'\in \{-D,\ldots, D\}$ there are no solutions to the equation $da_k=d'a_\ell$ satisfying $\max\{k,\ell\}> k_0$ and $k\neq \ell$, so condition \ref{item:cond-number-of-solutions} holds even with $O(1)$ instead of $o(n)$.
		
		Since $f$ is bounded, there exists  $A>0$ such that $|X_j|\le A$ a.s. Thus,
		 Taylor's formula for the exponential function and equation \eqref{eq:MDP-large-gaps-pom1}  imply (we write pw for pairwise)
		\begin{align*}
			\EE  &e^{\theta\sqrt{z_n}(S_n-S_{k_0})} 
			= \EE \prod_{j=k_0+1}^n e^{\theta\sqrt{z_n} X_j}
			= \EE \prod_{j=k_0+1}^n \Bigl( 1+\theta\sqrt{z_n}X_j +\frac{\theta^2z_n}2 X_j^2 + O(|z_n|)^{3/2} \Bigr)
			\\
			& =  \sum_{ \substack{k,\ell,m\in\{0,1,\ldots\} \\ k+\ell+m\le n-k_0}} \theta^{k+2\ell}z_n^{k/2 +\ell} 2^{-\ell} O(|z_n|)^{3m/2} {n-k_0-k-\ell \choose m} 
			\\
			& \hspace{1,5cm} \cdot \sum_{\substack{i_1,\ldots, i_k \in \{k_0+1,\ldots, n\} \\ \text{pw. distinct}}}\sum_{\substack{j_1,\ldots, j_\ell \in \{k_0+1,\ldots, n\} \\ \text{pw. distinct}}} \EE (X_{i_1}\cdots X_{i_k}X_{j_1}^2\cdots X_{j_\ell}^2)
			\\
			& \hspace{-7pt}\stackrel{\eqref{eq:MDP-large-gaps-pom1}}{=}  \sum_{ \substack{\ell,m\in\{0,1,\ldots\} \\ \ell+m\le n-k_0}} \Bigl(\frac{\theta^2z_n}2\Bigr)^{\ell} O(|z_n|)^{3m/2} {n-k_0-\ell \choose m} 
			\sum_{\substack{j_1,\ldots, j_\ell \in \{k_0+1,\ldots, n \}\\ \text{pw. distinct}}} \EE X_{j_1}^2\cdots \EE X_{j_\ell}^2
			\\ & =  \sum_{ \substack{\ell,m\in\{0,1,\ldots\} \\ \ell+m\le n-k_0}}   {n-k_0 \choose \ell}  {n-k_0-\ell \choose m}
 \Bigl(\frac{\theta^2z_n\EE X_1^2}2\Bigr)^{\ell} O(|z_n|)^{3m/2}.
 		\end{align*}	
		Moreover, we have
		\begin{align*}
			 e^{(n-k_0)\theta^2z_n  \EE X_1^2 /2}  &=  \biggl( 1+\frac{\theta^2z_n\EE X_j^2}2 + O(z_n^{2}) \biggr)^{n-k_0}
			\cr
			&= \sum_{ \substack{\ell,m\in\{0,1,\ldots\} \\ \ell+m\le n-k_0}}   {n-k_0 \choose \ell}  {n-k_0-\ell\choose m}
 \Bigl(\frac{\theta^2z_n\EE X_1^2}2\Bigr)^{\ell} O(|z_n|)^{2m}.
		\end{align*}
		 Since $z_n\to 0$ as $n\to\infty$, we therefore obtain
		 \begin{align*}	
		 	& \EE  e^{\theta\sqrt{z_n}(S_n-S_{k_0})}  \cr
			&= e^{(n-k_0)\theta^2 z_n  \EE X_1^2 /2} 
			+  \sum_{ \substack{\ell,m\in\{0,1,\ldots\} \\ \ell+m\le n-k_0}}   {n-k_0 \choose \ell}  {n-k_0-\ell \choose m}
 \Bigl(\frac{\theta^2z_n\EE X_1^2}2\Bigr)^{\ell} O(|z_n|)^{3m/2}
 			\\ &=e^{(n-k_0)\theta^2 z_n \EE X_1^2 /2}  +  \Bigl( 1+\frac{\theta^2z_n\EE X_1^2}2 + O(|z_n|)^{3/2} \Bigr)^{n-k_0}.
		 \end{align*}
		 Assumptions $|X_i|\le A$ a.s.,  $nz_n\to \infty$ and $z_n\to 0$ as $n\to\infty$, and the equation above yield
		 \begin{align*}
		 	  \frac1{nz_n} \log \EE e^{\theta \sqrt{z_n}S_n} 
			& = \frac1{nz_n} \log \EE (e^{\theta\sqrt{z_n}S_{k_0}}e^{\theta \sqrt{z_n}(S_n - S_{k_0})})  
			\cr &
			= \frac{1}{nz_n} \log \big( \EE e^{\theta \sqrt{z_n}S_n-{S_{k_0}}} \big) +o(1)
			\cr
			& = 
			\frac{n-k_0}{nz_n} \max\Bigl\{\frac{\theta^2 z_n \EE X_1^2}2
			,\log\Bigl( 1+\frac{\theta^2z_n\EE X_1^2}2 + O(|z_n|)^{3/2}\Bigr) \Bigr\}
			+o(1)
			\cr &
			= \frac{\theta^2\EE X_1^2}2 +o(1),
		 \end{align*}
		 so the desired identity \eqref{eq:aim-GE} holds.
		 
		 Now we move on to the proof of the identity $\EE S_n^2 = n\EE X_1^2 + O(1)$. By \eqref{eq:var-pom} it suffices to prove that condition \ref{item:cond-number-of-solutions} holds with $O(1)$ instead of $o(n)$, what we have shown before in this proof. 	
		 \end{proof}

\section{Mod-Gaussian convergence}	\label{sect:mod-G}

The theory of mod-$\phi$ convergence is a rather recent natural extension of the previously investigated types of convergence for some sequences of random variables. This notion of convergence encodes quite a lot of information reaching from local limit theorems up to moderate and even large deviations. It has turned out to be a powerful tool in analytic number theory, random matrix theory, and probability theory and we refer the reader to the monograph \cite{Mod-phi-book} by F\'eray, M\'eliot, and Nikeghbali for more information and additional references. In this section we are interested in a special type of mod-$\phi$ convergence, where $\phi$ is the standard Gaussian distribution. In this case we speak of mod-Gaussian convergence.

For the convenience of the reader, we present the definition; by $\mathfrak{R}(z)$ we shall denote the real part of a complex number $z$.

\begin{definition}
Let $(X_n)_{n\in\NN}$ be a sequence of real-valued random variables and assume that their moment generating functions $\varphi_n(z)\coloneqq \EE e^{zX_n}$ exist in a strip
    \[
      S_{(c,d)} \coloneqq \{ z\in \CC \,:\, \mathfrak{R}(z) \in (c,d) \}
    \] 
for some $c<d$, where we allow $c=-\infty$ and $d=+\infty$.  We assume further that there exists a non-constant infinitely divisible probability distribution $\phi$ with moment generating function 
    \[
      \int_{\RR} e^{zx} \phi(\text{d}x) = e^{\eta(z)}
    \]
that is well-defined on $S_{(c,d)}$, and an analytic function $\psi$ that does not vanish on $\mathfrak R(S_{(c,d)})$ such that the following holds: there exists a sequence $(t_n)_{n\in\NN}$ of real numbers converging to $+\infty$ such that
    \[
      e^{-t_n \eta(z)} \varphi_n(z) \stackrel{n\to\infty}{\longrightarrow} \psi(z) \qquad\text{locally uniformly on $S_{(c,d)}$}.
    \]   
We then say that the sequence $(X_n)_{n\in\NN}$ converges mod-$\phi$ on $S_{(c,d)}$ with parameters $(t_n)_{n\in\NN}$ and limiting function $\psi$.  If $\phi$ is the standard normal distribution, then $\eta(z)=z^2/2$, and we speak of mod-Gaussian convergence.
\end{definition}

We are now going to prove mod-Gaussian convergence for appropriately normalized partial sums $(S_n^f)_{n\in\NN}$ under some additional growth assumptions on the lacunary sequence $(a_n)_{n\in\NN}$.

Recall that $\gamma_m(X_1)$ is the $m$-th cumulant of the random variable $X_1$. 
Since we assume that $c_0=0$, $\gamma_1(X_1)=0$ and $\gamma_2(X_1)=\EE X_1^2 = \sum_{d=-D}^D c_d^2 >0$.
Let $\rho \ge 3$ be an integer such that 
	\[
	\gamma_3(X_1),\ldots ,\gamma_{\rho-1}(X_1)=0 \quad \text{and}\quad \gamma_{\rho}(X_1)\neq 0.
	\]
The rate of mod-Gaussian convergence depends on the parameter $\rho$. Note that in the classical case when $X_n = \cos(2\pi a_n U)$, $n\in\NN$, we have $\rho=4$ and 
      $
         \gamma_4(X_1)=\EE X_1^4 - 3(\EE X^2)^2 =3/8-3/4= -3/8.
      $  
		
We are able to prove the Mod-Gaussian convergence for lacunary sequences satisfying one of the following growth conditions:
	\begin{gather}
		\frac{a_{n+1}}{n^{\rho-1} a_{n}} \mathop{\longrightarrow}^{n\to \infty} \infty,
		\tag{Condition 1}
		\label{condition1}
	\\
		\frac{a_{n+1}}{a_{n}}\in \NN  \quad \text{for all } n\in \NN 
		 \quad \textbf{and} \quad \text{there exists a sequence } (x_n)_{n\in\NN}
		 \tag{Condition 2}
		\label{condition2}
		 \\ 
		  \text{ of positive real numbers converging to }0, \text{ such that } 
		 \sum_{k=\lceil n^{1/\rho} x_n\rceil}^n \frac{a_{k}}{a_{k+1}} = o(n^{1/\rho}) .
		\nonumber
	\end{gather}
	
	Note that \eqref{condition2} is implied by the following condition:
	\begin{gather}
		\frac{a_{n+1}}{n^{1-1/\rho}a_{n}} \mathop{\longrightarrow}^{n\to \infty} \infty \qquad \textbf{and}\qquad \frac{a_{n+1}}{a_{n}}\in \NN  \quad \text{for all } n\in \NN .
		\tag{Condition 3}
		\label{condition3}
	\end{gather}
	\begin{proof}[Proof that \eqref{condition3} implies \eqref{condition2}]
		Assume that $\frac{a_{n+1}}{n^{1-1/\rho}a_{n}} \to  \infty$. Then
		\begin{align*}
			\sum _{k=\lceil n^{1/2\rho}\rceil}^n \frac{a_{k}}{a_{k+1}} 
			& = \sum_{k=\lceil n^{1/2\rho}\rceil}^n \frac{k^{1-1/\rho}a_{k}}{a_{k+1}} \cdot\frac{1}{k^{1-1/\rho}} 
			 = o(1)  \sum_{k=\lceil n^{1/2\rho}\rceil}^n  k^{1/\rho-1}
			= o(n^{1/\rho}),
		\end{align*}
		so  \eqref{condition2} holds with $x_n=n^{-1/2\rho}$, $n\in\NN$.
	\end{proof}
	
	\begin{theorem}[Mod-Gaussian convergence for $(\frac{S_n^f}{n^{1/\rho}})_{n\in\NN}$]	\label{prop:Mod-Gaussian-conditions-1-2}
		Assume that $(a_n)_{n\in\NN}$ satisfies \eqref{condition1} or \eqref{condition2}. 
		Then 
		\[
			\exp\Bigl(-\frac{\theta^2n^{1-2/\rho}\EE X_1^2}{2}\Bigr)  \EE \exp\Bigl(\theta \frac{S_n^f}{n^{1/\rho}}\Bigr)  \mathop{\longrightarrow}^{n\to \infty}  \exp \Bigl(\frac{\theta^\rho\gamma_\rho}{\rho!}\Bigr)
		\]
		and
		\[
			\exp\Bigl(-\frac{\theta^2\sigma_n^2}{2n^{2/\rho}}\Bigr)  \EE \exp\Bigl(\theta \frac{S_n^f}{n^{1/\rho}}\Bigr)  \mathop{\longrightarrow}^{n\to \infty}  \exp \Bigl(\frac{\theta^\rho\gamma_\rho}{\rho!}\Bigr)
		\]
		locally uniformly on $\CC$, where $\sigma_n^2 =( \sigma_n^f )^2=\EE (S_n^f)^2$ and $\gamma_\rho = \gamma_\rho(X_1)$.
	\end{theorem}

\begin{proof}[Proof in the case of \eqref{condition2}]
	Let $S_n=S_n^f$  and $T_n = \sum_{k=1}^n Y_k$, where $Y_1, Y_2,\ldots$ are i.i.d.~copies of $X_1$.
	It follows from Theorem~\ref{thm:MDP-large-gap} that $\sigma_n^2=\EE S_n^2=n\EE X_1 +O(1)$ and so it suffices to prove the first asserted convergence.
	Fix an arbitrary $R>0$. We shall prove the uniform convergence on the disk $D_R\coloneqq \{\theta \in \CC\colon |\theta|\le R\}$.
	
	Let $(x_n)_{n\in\NN}$ be a sequence of positive real numbers satisfying \eqref{condition2}, and let $k_0=k_0(n)\coloneqq \lceil n^{1/\rho} x_n\rceil$.   
	Since the trigonometric polynomial $f$ is bounded, there exists a real number $A>0$ such that $|X_k|\le A$  a.s., so
	\begin{equation}
		|n^{-1/\rho} S_{k_0}|	= \Bigl| n^{-1/\rho} \sum_{k=1}^{k_0} X_k\Bigr| \le n^{-1/\rho} k_0 A  \le (x_n+n^{-1/\rho} )A= o(1) \qquad \text{a.s.}
		\label{eq:first-k0-steps1}
	\end{equation}
	Thus, we  only need to investigate the asymptotic behavior of 
	\[
	\exp\Bigl(-\frac{\theta^2n \EE X_1^2}{2n^{2/\rho}}\Bigr) \EE \exp\Bigl(\theta \frac{S_n-S_{k_0}}{n^{1/\rho}}\Bigr).
	\]
	In order to do that, we repeat (with  minor changes only) the proof from \cite[Section~3.1]{agkpr2020} of \cite[Theorem A]{agkpr2020}   in a simple special case.  The first change appears in the first display after \cite[inequality~(3.1)]{agkpr2020}: since $x\mapsto f(a_kx)$ has Lipschitz constant bounded by $2\pi a_k \sum_{d=-D}^D |c_d|\eqqcolon a_kL$,  \eqref{condition2} implies that
	\begin{equation}
		\|(S_n -S_{k_0} )-(T_n-T_{k_0})\|_{\infty} \le L \sum_{k=\lceil n^{1/\rho} x_n\rceil+1}^n \frac{a_{k}}{a_{k+1}} = o(n^{1/\rho}) \qquad \text{a.s.}
		\label{eq:dif-sup-ST}
	\end{equation}
Moreover, since $T_n-T_{k_0} = Y_{k_0+1}+\cdots +Y_n$ is a sum of i.i.d.~random variables with the same distribution as $X_1$, we have
	\begin{equation} \label{eq:T_n-independent}
		\EE  \exp\Bigl(\theta \frac{T_n-T_{k_0}}{n^{1/\rho}}\Bigr) =  \bigl(\EE \exp(\theta n^{-1/\rho} X_1) \bigr)^{n-k_0}. 
	\end{equation}
Recall that $\gamma_1(X_1)=\EE X_1=c_0=0$. Since $X_1$ is bounded a.s., $\EE e^{tX_1}<\infty$ for every $t>0$. Therefore, $u\mapsto \EE e^{uX_1}$ is entire and does not vanish in a neighbourhood of $0$, so $u\mapsto \log \EE e^{uX_1}$ is analytic in a neighbourhood of $0$ and
	\begin{equation}\label{eq:taylor-cumulants}
		\log \EE e^{uX_1} 
		= \sum_{m=1}^\infty \frac{u^m \gamma_m(X_1)}{m!} 
		= \frac{u^2\gamma_2(X_1)}{2} + \frac{u^\rho\gamma_\rho(X_1)}{\rho!} +\sum_{m=\rho+1}^\infty  \frac{u^m \gamma_m(X_1)}{m!} .
	\end{equation}
Since $|X_1|\le A$ a.s. and $\gamma_m(X_1)/m!$ are the coefficients in the analytic expansion of $\log  \EE e^{uX}$, we have (see, e.g., \cite[equation (1.34)]{SS-book})
	\begin{align*}
		|\gamma_m(X_1) |
		& =\biggl| \sum_{k=1}^m \frac{(-1)^{k+1}}k \sum_{\substack{m_1+\cdots+m_k=m \\ m_1,\ldots,m_k\ge 1}} \frac{m!}{m_1!\cdots m_k!} \EE X_1^{m_1} \cdots \EE X_k^{m_k}\biggr| 
		\\& \le 
		A^m  \sum_{k=1}^m \frac{1}k \sum_{\substack{m_1+\cdots+m_k=m \\ m_1,\ldots,m_k\ge 0}} \frac{m!}{m_1!\cdots m_k!}		
		\\ &  = A^m \sum_{k=1}^m k^{m-1} \le A^mm^m \le (Ae)^m m!.
	\end{align*}
Thus, for every $\eps<1$, we have, uniformly on $|u|\le \eps(eA)^{-1}$, 
	\[
	\biggl| \sum_{m=\rho+1}^\infty  \frac{u^m \gamma_m(X_1)}{m!} \biggr| 
	\le  \sum_{m=\rho+1}^\infty \frac{|u|^m |\gamma_m(X_1)|}{m!}
	\le
	\frac{ \eps^{\rho+1}}{1-\eps}.
	\]
	This inequality applied with $u=\theta n^{-1/\rho}$, $\eps = n^{-2/(2\rho+1)}$ (so that $|u|\le \eps(eA)^{-1}$  for sufficiently large $n\in\NN$) and equation \eqref{eq:taylor-cumulants} yield
	\[
		\log \EE e^{\theta X_1/n^{1/\rho}} 
		= \frac{\theta^2\gamma_2(X_1)}{2n^{2/\rho}} + \frac{\theta^\rho\gamma_\rho(X_1)}{n\rho!} + o(1/n)
	\]
	uniformly on the disk $D_R$.
	
		Hence, \eqref{eq:T_n-independent} yields
	\begin{equation}
		\EE  \exp\Bigl(\theta \frac{T_n-T_{k_0}}{n^{1/\rho}}\Bigr) = \exp \Bigl(\frac{\theta^2 n\EE X_1}{2n^{2/\rho}}  + \frac{\theta^\rho \gamma_\rho}{\rho!}  \Bigr) \bigl(1+o(1)\bigr)
		\label{eq:LHS-for-indep-sums}
	\end{equation}
	uniformly on $D_R$.
	This  together with estimate \eqref{eq:dif-sup-ST} imply that
	\begin{align*}
		\exp\Bigl(&-\frac{\theta^2n\EE X_1}{2n^{2/\rho}}\Bigr)  \EE \exp\Bigl(\theta \frac{S_n-S_{k_0}}{n^{1/\rho}}\Bigr) 
		\\ &= \exp\Bigl(-\frac{\theta^2n\EE X_1}{2n^{2/\rho}}\Bigr)  \EE \exp\Bigl(\theta \frac{T_n-T_{k_0}}{n^{1/\rho}}\Bigr) \exp\Bigl(\theta \frac{(S_n -S_{k_0} )-(T_n-T_{k_0})}{n^{1/\rho}}\Bigr) 
		\\ &=  \exp\Bigl(\frac{\theta^\rho \gamma_\rho}{\rho!} \Bigr) \bigl(1+o(1)\bigr)
	\end{align*}
	uniformly on $D_R$. 
	This and \eqref{eq:first-k0-steps1} yield the assertion.
\end{proof}

	Before we move to the proof of Proposition~\ref{prop:Mod-Gaussian-conditions-1-2} in the case of \eqref{condition1}, we give a short proof to the following lemma.
	\begin{lemma}	\label{lem-exp-poly}
		Let $A\ge 1$. Then, for every $x\in \CC$ with $|x|\le A$, 
		\begin{equation*}
			\Bigl|e^x - \sum_{m=0}^{\lceil 2eA \rceil} \frac{x^m}{m!}\Bigr|  \le 2^{-A}.
		\end{equation*}
	\end{lemma}
	\begin{proof}
		Since 
		\[
			\Bigl|e^x - \sum_{m=0}^{\lceil 2eA \rceil} \frac{x^m}{m!}\Bigr| 
			\le \sum_{\lceil 2eA \rceil +1}^\infty \frac{|x|^m}{m!}
			\le \frac{A^{\lceil 2eA \rceil +1}}{(\lceil 2eA \rceil +1)!} \sum_{m=0}^\infty  \frac{|x|^m}{m!} 
			\le \frac{A^{\lceil 2eA \rceil +1} e^{A}}{(\lceil 2eA \rceil +1)!},
		\]
		the assertion follows from Stirling's formula.
	\end{proof}
		
\begin{proof}[Proof of Theorem~\ref{prop:Mod-Gaussian-conditions-1-2} in the case of \eqref{condition1}]
Let us denote $S_n=S_n^f$  and $T_n = \sum_{k=1}^n Y_k$, where $Y_1, Y_2,\ldots$ are i.i.d.~copies of the random variable $X_1$.
It follows from Theorem~\ref{thm:MDP-large-gap} that $\sigma_n^2=\EE S_n=n\EE X_1 +O(1)$ and thus it suffices to prove the first asserted convergence.
	Fix an arbitrary $R>1$. We shall prove the uniform convergence on the disk $D_R\coloneqq \{\theta \in \CC\colon |\theta|\le R\}$.
	
	Let $(y_n)_{n\in\NN}$ be a sequence  of positive numbers converging to $0$, and let $k_0=k_0(n,y_n)\coloneqq \lceil n^{1/\rho} y_n\rceil$. 	
	Since $f$ is bounded, there exists a real number $B>0$ such that $|X_k|\le B$  a.s. Therefore,
	\begin{equation}
		|n^{-1/\rho} S_{k_0}|	= \Bigl| n^{-1/\rho} \sum_{k=1}^{k_0} X_k\Bigr| \le Bn^{-1/\rho} k_0 \le B( y_n +n^{-1/\rho})= o(1) \qquad \text{a.s.}
		\label{eq:first-k0-steps2}
	\end{equation}

	 Let $M=M(n)=\lceil 2en^{1-1/\rho}R \rceil$. Note that  $|\theta \frac{S_n-S_{k_0}}{n^{1/\rho}} |\le BRn^{1-1/\rho}$ a.s. Thus, Lemma~\ref{lem-exp-poly} (applied with $A=BRn^{1-1/\rho}$) implies that
	\begin{equation}
		\EE \exp\Bigl(\theta \frac{S_n-S_{k_0}}{n^{1/\rho}}\Bigr) = \sum_{m=0}^{M} \frac{\theta^m\EE(S_n-S_{k_0})^m}{m! n^{m/\rho}} +r_n(\theta),
		\label{eq:exp-to-poly}
	\end{equation}
	where a random variable $r_n(\theta)$ almost surely satisfies $|r_n(\theta)|\le e^{-BT n^{1-1/\rho}\log 2 }$ uniformly on $D_R$.
	
	We shall find a sequence $(y_n)_{n\in\NN}$ of positive numbers converging to $0$ such that for sufficiently large $n\in\NN$ the following holds:
	\begin{equation}
		\EE(S_n-S_{k_0})^m=\EE(T_n-T_{k_0})^m \quad \text{for every integer } m\le 2eBRn^{1-1/\rho}+1.
		\label{eq:aim-go-to-indep}
	\end{equation}
	Having such a sequence we obtain by equation \eqref{eq:exp-to-poly} and   Lemma~\ref{lem-exp-poly} that
	\[
		\exp\Bigl(-\frac{\theta^2 n\EE X_1^2}{2n^{2/\rho}}\Bigr)  \EE \exp\Bigl(\theta \frac{S_n-S_{k_0}}{n^{1/\rho}}\Bigr) 
		= \exp\Bigl(-\frac{\theta^2 n\EE X_1^2}{2n^{2/\rho}}\Bigr)  \EE \exp\Bigl(\theta \frac{T_n-T_{k_0}}{n^{1/\rho}}\Bigr)+ o(1)
	\]
	uniformly on $D_R$.
	 Hence, equation \eqref{eq:LHS-for-indep-sums} implies 
	 	 \[
	 \exp\Bigl(-\frac{\theta^2 n\EE X_1^2}{2n^{2/\rho}}\Bigr)  \EE \exp\Bigl(\theta \frac{S_n-S_{k_0}}{n^{1/\rho}}\Bigr) = \exp\Bigl(\frac{\theta^\rho \gamma_\rho}{\rho!}\Bigr) \bigl(1+ o(1)\bigr)
	 \]
	 uniformly on $D_R$,
	 which together with \eqref{eq:first-k0-steps2} yields the assertion. Therefore, we only need to find the sequence $(y_n)_{n\in\NN}$ of positive numbers, converging to $0$, and satisfying property \eqref{eq:aim-go-to-indep} (recall that the dependence on $(y_n)_{n\in\NN}$ is hidden in $k_0=k_0(n,y_n)$).
	 
	It follows from formula \eqref{eq:moms-Sn} that 
	 \[
		\EE(S_n-S_{k_0})^m = \sum_{\substack{\ell_1,\ldots, \ell_m \in \{k_0+1,\ldots ,n\} \\ d_1,\ldots, d_m \in \{-D,\ldots,D\} }} c_{d_1}\cdots c_{d_m} \ind_{\{ \sum_{j=1}^m d_j a_{\ell_j} = 0\}}.
	 \]
	We call a solution to the equation  $\sum_{j=1}^m d_j a_{\ell_j}=0$ (with unknown $d_1, \ldots , d_m \in \{-D,D\}$ and $\ell_{1},\ldots, \ell_m \in \{k_0+1,\ldots , n\}$) \textit{trivial}, if for every $k\in \{k_0+1,\ldots , n\}$, we have $\sum_{j=1}^m d_j \ind_{\{\ell_j=k\}} = 0$. Formula \eqref{eq:moms-Sn-indep}  implies that 
	\[
		\sum_{\substack{\ell_1,\ldots, \ell_m \in \{k_0+1,\ldots ,n\} \\ d_1,\ldots, d_m \in \{-D,\ldots,D\} \\ \text{trivial} }} c_{d_1}\cdots c_{d_m} = \EE(T_n-T_{k_0})^m ,
	 \]
	 where the sum runs over all trivial solutions to the equation $\sum_{j=1}^m d_j a_{\ell_j}=0$. This means that \eqref{eq:aim-go-to-indep} is satisfied  if there are no nontrivial solutions to the equation $\sum_{j=1}^m d_j a_{\ell_j}=0$ with $d_1, \ldots , d_m \in \{-D,D\}$ and $\ell_{1},\ldots, \ell_m \in \{k_0+1,\ldots , n\}$.
	 
	 Since $\frac{a_{n+1}}{n^{\rho-1} a_{n}}\to \infty$, there exists  a decreasing sequence $(x_n)_{n\in\NN}$ of positive numbers converging to $0$, such that $\frac{a_{n+1}}{n^{\rho-1}a_{n}} x_n \to \infty$ and $n^{\rho-1}x_n \to \infty$. Let $y_n=x_{\lceil n^{1/2\rho} \rceil}^{1/(\rho-1)}$. 
	 
	 Assume contrarily, that for an arbitrarily large $n\in\NN$ there exists a nontrivial solution to the equation $\sum_{j=1}^m d_j a_{\ell_j}=0$ with $d_1, \ldots , d_m \in \{-D,D\}$, $\ell_{1},\ldots, \ell_m \in \{k_0+~1,\ldots , n\}$, and $m\le 2eBRn^{1-1/\rho}+1$. Let $k_1$ be the largest number $k\in \{k_0+1,\ldots , n\}$ for which $\sum_{j=1}^m d_j \ind_{\{\ell_j=k\}} \neq 0$ (such a number exists, since our solution is not trivial). Since $k_1$ is maximal, we know that
	 \begin{align*}
	 	0  = \Bigl| \sum_{j=1}^m d_j a_{\ell_j}  \Bigr|
		& =	\Bigl| \sum_{k=k_0+1}^n a_k\sum _{j=1}^m d_j \ind_{\{\ell_j=k\}}  \Bigr|
		\\
		& = 	\Bigl| \sum_{k=k_0+1}^{k_1-1} a_k\sum _{j=1}^m d_j \ind_{\{\ell_j=k\}} + a_{k_1} \sum_{j=1}^m d_j \ind_{\{\ell_j=k_1\}}  \Bigr|
		\\ & 
		\ge a_{k_1} \Bigl| \sum_{j=1}^m d_j \ind_{\{\ell_j=k_1\}}  \Bigr| - \Bigl| \sum_{k=k_0+1}^{k_1-1} a_k\sum _{j=1}^m d_j \ind_{\{\ell_j=k\}} \Bigr|
		\\
		& \ge a_{k_1} -  \sum_{k=k_0+1}^{k_1-1} \sum _{j=1}^m |d_j|  \ind_{\{\ell_j=k\}}  \max_{k_0\le k\le k_1-1}a_{k}
		\\
		& \ge  a_{k_1} - mDa_{k_1-1} 
		\ge  a_{k_1} - (2en^{1-1/\rho}BR+1)D a_{k_1-1}.
	 \end{align*}
	 Therefore, in order to get a contradiction it suffices to prove that for all sufficiently large integers $n$, 
	 \begin{equation}
	 	a_{k+1} > (2en^{1-1/\rho}BR+1)D a_{k} \quad \text{for every } k\in  \{\lceil n^{1/\rho}y_n \rceil ,\ldots , n\}.
	\label{eq:aim-big-gaps}
	\end{equation}
	 
	Recall that $n^{\rho-1}x_n \to \infty$, so $n^{1/{2\rho}}y_{n} \to \infty$. In particular, for $n\in\NN$ large enough $n^{1/\rho}y_n\ge n^{1/2\rho}$. 
	Thus, $k \ge n^{1/\rho}y_n \to \infty$, so $\frac{k^{\rho-1} a_k }{a_{k+1}x_k} =o(1)$. 
	Moreover, since $y_n^{\rho-1}=x_{\lceil n^{1/2\rho}\rceil}$,  and $(x_n)_{n\in\NN}$ is decreasing, we have for every $k\ge n^{1/\rho} y_n$ and sufficiently large $n\in\NN$,
	 \begin{align*}
	 	\frac{n^{1-1/\rho} a_k}{a_{k+1}} 
		& = \frac{k^{\rho-1} a_k }{a_{k+1}x_k} \cdot \frac{n^{1-1/\rho}}{k^{\rho-1}} \cdot x_k 
		= o(1) \cdot \frac{1}{y_n^{\rho-1}} x_k 
		\\
		&=  o(1) \cdot \frac{x_k}{x_{\lceil n^{1/2\rho}\rceil}} 
		\le  o(1) \cdot \frac{x_{\lceil n^{1/\rho}y_n\rceil }}{x_{\lceil n^{1/2\rho} \rceil}}
		 \le o(1) \cdot \frac{x_{\lceil n^{ 1/2\rho}\rceil }}{x_{\lceil n^{1/2\rho}\rceil}} = o(1).
	 \end{align*}
	 This confirms \eqref{eq:aim-big-gaps} and thus completes the proof.
\end{proof}

\subsection*{Acknowledgement} JP is supported by the Austrian Science Fund (FWF) under project P32405 and by the German Research Foundation (DFG) under project 516672205. MS was supported by the Austrian Science Fund (FWF) under project P32405 and by the  National Science Center, Poland, via the Maestro grant no. 2015/18/A/ST1/00553.

  \bibliographystyle{amsplain}
  \bibliography{lacunary_sums}

\end{document}